\documentclass[reqno,oneside]{amsart}
\usepackage{hyperref}
\usepackage{geometry}
\usepackage[ansinew]{inputenc}
\usepackage{graphicx}
\usepackage{amsmath}
\usepackage{amsthm}
\usepackage{amssymb, color}%
\usepackage[numbers, square]{natbib}
\usepackage{mathrsfs}
\usepackage{bbm}
\usepackage{pgf, tikz}
\usetikzlibrary{trees}
\usepackage{mathptmx}
\usepackage{subcaption}

\usepackage[numbers,square]{natbib}
\usepackage{color}
\allowdisplaybreaks

\usepackage{bbm}
\usepackage[all]{xy}
\usepackage{amscd}
\usepackage{stmaryrd}
\usepackage{verbatim}

\usepackage[english]{babel}

\newtheorem{thm}{Theorem}[section]
\newtheorem{lem}[thm]{Lemma}

\theoremstyle{remark}
\newtheorem{rem}[thm]{Remark}
\theoremstyle{definition}

\numberwithin{equation}{section}

\def\P{{\mathbb{P}}}
\def\R{{\mathbb{R}}}

\newcommand{\ind}{\mathbbm{1}}
\newcommand{\EE}{\mathbb{E}}
\newcommand{\PP}{\mathbb{P}}
\newcommand{\E}{\mathbb{E}}

\renewcommand{\P}{\mathbb{P}}
\newcommand{\NN}{\mathbb{N}}
\newcommand{\N}{\mathbb{N}}

\newcommand{\RR}{\mathbb{R}}
\newcommand{\C}{\mathbb{C}}
\newcommand{\XX}{\mathbb X}
\newcommand{\Fcal}{\mathcal{F}}

\newcommand{\Zcal}{\mathcal{Z}}

\newcommand{\Ccal}{\mathcal{C}}

\newcommand{\Tcal}{\mathcal{T}}

\newcommand{\Xcal}{\mathcal{X}}
\newcommand{\Qcal}{\mathcal{Q}}
\newcommand{\Wcal}{\mathcal{W}}

\newcommand{\ud}{\mathrm{d}}

\renewcommand{\ud}{{\rm d}}

\usepackage{soul}

\hypersetup{
    colorlinks = false,
citebordercolor= {white},
    linkbordercolor = {white}
}

  \evensidemargin
\oddsidemargin

\begin{document}

	\title{Solutions of kinetic-type equations with perturbed collisions}
	\author[D.~Buraczewski, P.~Dyszewski, A.~Marynych]{Dariusz Buraczewski, Piotr Dyszewski and Alexander Marynych}
	\thanks{DB and PD were supported by the National Science Center, Poland (Opus, grant number 2020/39/B/ST1/00209).}
	\keywords{Additive martingale; branching random walk; inhomogeneous smoothing transform; Kac model; kinetic equation; random trees}
	\subjclass[2020]{Primary: 60J85; Secondary: 82C40, 60F05.}

	\begin{abstract}
		We study a class of kinetic-type differential equations $\partial \phi_t/\partial t+\phi_t=\widehat{\mathcal{Q}}\phi_t$, where
		$\widehat{\mathcal{Q}}$ is an inhomogeneous smoothing transform and, for every $t\geq 0$, $\phi_t$ is the Fourier--Stieltjes transform of a probability measure. We show that under mild assumptions on
		$\widehat{\mathcal{Q}}$ the above differential equation possesses a unique solution and represent this solution as the characteristic function of a certain stochastic process associated with the continuous time branching random walk pertaining to $\widehat{\mathcal{Q}}$. Establishing limit theorems for this  process allows us to describe asymptotic properties of the solution, as $t\to\infty$.
	\end{abstract}

	\maketitle

\section{Introduction}\label{sec:Intro}

	The article is devoted to the study of generalized kinetic-type equations of the form
	\begin{equation}\label{eq:1:KineticDistr}
		\left\{
		\begin{array}{l}
		\partial_t\nu_t  = \Qcal (\nu_t) - \nu_t,\\
		\nu_0 =\nu;
		\end{array}
		\right.
	\end{equation}	
	where $(\nu_t)_{t \geq 0}$ is a family of probability measures on $\RR$. The derivative with respect to the time parameter $t$ is understood in the weak sense, that is, $\partial_t\nu_t$ is a signed measure such that $\frac{\ud}{\ud t}\int g(y) \nu_t(\ud y)=\int g(y) \partial_t\nu_t(\ud y)$, for every bounded continuous $g:\R \to \C$. The operator $\mathcal{Q}$ on the right-hand side is
	a~smoothing transform defined in the following fashion.
	For a given random vector of coefficients $(N,C, A_1, A_2, \ldots )$ taking values in $\{0,1,2,\ldots\}\times \R\times (0,\infty)^{\N}$ and a probability measure $\nu$ on $\RR$, take a sequence of independent identically
	distributed (i.i.d.) random variables $Y_1, Y_2, \ldots$ with common distribution $\nu$ which are also independent of the vector $(N,C, A_1, A_2, \ldots )$. The image $\mathcal{Q}(\nu)$ is then defined by
	\begin{equation*}
		\mathcal{Q}(\nu) = {\rm Law} \left(  \sum_{k =  1}^N A_k Y_k+C \right),
	\end{equation*}
	where ${\rm Law}(X)$ denotes the law of a random variable $X$.
	
	The equations of the form~\eqref{eq:1:KineticDistr} originate from the Boltzmann equation. In the Kac caricature of a Maxwell
	gas~\cite{Kac:1956}, $\nu_t$ is the distribution of velocity of
	a typical particle in a homogeneous gas exposed to
	elastic binary collisions. This model corresponds to the choice of parameters $N=2$, $C=0$, $(A_1, A_2) = (\sin \Theta, \cos \Theta)$, with $\Theta$ being a random variable with the uniform distribution on $[0,2\pi]$.
	The term $\mathcal{Q}(\nu_t) - \nu_t$ on the right-hand side of~\eqref{eq:1:KineticDistr} describes the change in the velocity after a binary collision. Observe that here $A_1$ and $A_2$ are not positive, which is the standing assumption for the present work. However, it is known that one can reduce the Kac's model to a model with positive coefficients and a symmetric initial condition; see, for example, footnote on p.~2 in~\cite{Buraczewski2019}. The inelastic version of the caricature, that is, when $\E\left[A_1^2+A_2^2\right]\neq 1$, has been introduced 	in~\cite{Pulvirenti:2004} and studied later in~\cite{bassetti2008probabilistic}.  Equation~\eqref{eq:1:KineticDistr} in the general form involving the smoothing transform appears  for the first time in~\cite{bassetti2011central}, where the case $\P[C=0]=1$ and $N=2$ was investigated. In~\cite{bassetti2011central} it was also explained how various kinetic equations developed for describing market economy \cite{matthes2008analysis,matthes2008steady} fit into the above framework.  The perturbed case $\P[C=0]<1$, to the best of our knowledge, was introduced and studied only in~\cite{Bassetti2011} in the binary case $N=2$.
		
The classical choice $N=2$ is in general motivated by physical applications where it is natural to assume that every collision involves just two particles.
Nevertheless, the case of non-binary collisions, that is, when $\P[N=2]<1$, being of interest from a mathematical viewpoint, has also received some attention in the literature.
The case of deterministic $N\geq 2$ has been treated in  the papers~\cite{bassetti:2012:self,bobylev2009self,Bogus2020}, whereas in the recent work~\cite{Buraczewski2019} the variable $N$ is also allowed to be random. In all aforementioned references (except~\cite{Bassetti2011}) it is assumed that $\P[C=0]=1$. Besides pure mathematical interest, a motivation for the analysis of kinetic equations of the form~\eqref{eq:1:KineticDistr} in a general situation of random $N$ and non-vanishing $C$ arises in connection with models for wealth redistribution in econophysics. We hinge on approach and terminology borrowed from~\cite{Bassetti2011,bisi:2009,Buraczewski2019} and consider a class of models with indistinguishable trading agents. The $k$-th agent state is characterized by his current wealth $w_k\geq 0$.
Unlike in physical applications, it is natural to assume that interactions (trades) may involve a random number $N$ of agents. Conditional on the event $\{N=n\}$, $n\in\mathbb{N}$, an interaction has the form
$$
w^{\ast}_k=\sum_{j=1}^{n}p_{k,j}w_j+c_k,\quad k=1,\ldots,n,
$$
where $w_k$ (respectively, $w^{\ast}_k$) is the  pre-trade (respectively, post-trade) wealth of the $k$-the agent participating in the interaction. Here the coefficients $p_{k,j}$ are assumed to be random and represent the redistribution of the wealth between  agents, whereas $c_k$ represents, say, taxation operated by an external subject. This model is described by~\eqref{eq:1:KineticDistr} with the following distribution of parameters:
$$
\P[(C,A_1,A_2,\ldots,A_n)\in\cdot |N=n]=\frac{1}{n}\sum_{k=1}^{n}\P[(c_k,p_{k,1},\ldots,p_{k,n})\in\cdot],\quad n=0,1,2,\ldots.
$$
	
	Our aim is to provide robust probabilistic tools to study asymptotic properties of the solutions to~\eqref{eq:1:KineticDistr}.
	For convenience, we shall formulate and prove our main results in the terms of the characteristic functions of $(\nu_t)_{t \geq 0}$ which will be denoted via
	\begin{equation*}
		\phi_t(\xi)  = \phi(t,\xi) =  \int_{\mathbb{R}} e^{i \xi x} \: \nu_t(\ud x),\quad t \geq 0, \quad \xi \in \R.
	\end{equation*}	
	The relation~\eqref{eq:1:KineticDistr} written in terms of $(\phi_t)_{t \geq 0}$ reads
	\begin{equation}\label{eq:2:KineticChar}
		\left \{ \begin{array}{l} \partial_t \phi_t(\xi)  = \widehat{\mathcal{Q}}\phi_t(\xi) - \phi_t(\xi), \\ \phi_0(\xi)  = \phi(\xi); \end{array} \right.
	\end{equation}
	where $\phi(\xi)  = \int_{\mathbb{R}} e^{i \xi x} \: \nu(\ud x)$ and $\widehat{\mathcal{Q}}$ is the functional operator associated with the smoothing transform $\mathcal{Q}$. More precisely, for a
	continuous function $\psi \colon \R \to \C$ taking
	values in the unit disk $\{ z \in \C \: : \: \|z \| \leq 1\}$ we
	define the function $\widehat{\mathcal{Q}}\psi \colon \R \to \C$  via
	\begin{equation}\label{eq:inhom_smooth_def}
		\widehat{\mathcal{Q}}\psi (\xi) = \E \left[ e^{i \xi C} \: \prod_{k=1}^N \psi(A_k\xi) \right],  \quad \xi \in \R.
	\end{equation}
In what follows we denote by $R$ a generic random variable with the distribution
$\nu=\nu_0$, so $\phi(\xi)=\E e^{i \xi R}$.

	We shall show below, see Theorem~\ref{mthm1}, that under mild assumptions on $(N,C, A_1, A_2, \ldots )$, equation~\eqref{eq:2:KineticChar} for an arbitrary initial condition $\phi_0$ possesses a unique solution in the class of functions
\begin{multline}\label{eq:class_of_functions}
\Big\{\phi:[0,\infty)\times\mathbb{R}\mapsto\mathbb{C}\Big| \phi(t,\cdot)\text{ for every }t\geq 0 \text{ is the characteristic function of a probability measure,}\\
 \phi(\cdot,\xi)\in C^{(1)}([0,\infty)\text{ for every }\xi\in\mathbb{R}\Big\}.
\end{multline}
	Having settled uniqueness we then turn our attention to asymptotic analysis of this solution, as $t\to\infty$. Depending on a choice of the distribution of the vector of coefficients $(N,C, A_1, A_2, \ldots )$ we consider the
	two following scenarios.
	
	In the first regime $\phi_t(\xi)\to \phi_{\infty}(\xi)$, as $t\to\infty$, for some non-degenerate limit function $\phi_\infty$ and all $\xi\in\R$. Then one expects that $\phi_\infty$
	satisfies the stationary version of~\eqref{eq:2:KineticChar} that is
	\begin{equation*}
		 \widehat{\mathcal{Q}}\phi_\infty(\xi) = \phi_\infty(\xi),\quad \xi\in\R.
	\end{equation*}
	From the probabilistic point of view one would interpret $\phi_\infty$ as the Fourier--Stieltjes transform a fixed point of the inhomogeneous smoothing transform $\mathcal{Q}$. As verified by~\cite[Proposition 3]{Bassetti2011} in the case
	$\P[N=2]=1$ and Theorem~\ref{mthm4} of the present article
	for general $N$,  $\phi_\infty$ is the characteristic function of a random variable $X$ which satisfies
	\begin{equation*}
		{\rm Law}(X)= {\rm Law} \left(\sum_{k=1}^N A_kX_k + C\right),
	\end{equation*}
	with $(X_k)_{k\geq 1}$ denoting independent copies of $X$ which are also independent of the vector $(N,C, A_1, A_2, \ldots )$.
	
	In the second scenario $\phi_t$ has a nondegenerate pointwise limit only after a suitable rescaling. For some fixed $\alpha, \beta \in \RR$, consider a rescaled function
	\begin{equation*}
		w_t(\xi) = \phi_{t} \left( (t+1)^\beta e^{\alpha t} \xi \right), \quad t\geq 0,\quad \xi \in \RR.
	\end{equation*}
	Note that the factor $(t+1)^{\beta}$ instead of a more natural and asymptotically equivalent  $t^{\beta}$ is taken for convenience, since in this case $w_0(\xi)=\phi_0(\xi)=\phi(\xi)$. For the sake of simplicity assume for the time
	being that $\int_{\R}x\nu_{t}({\rm d}x)<\infty$, for all $t\geq 0$. This condition ensures that the function $w_t(\xi)$ is differentiable with respect to $\xi$ for all $t\geq 0$, and equation~\eqref{eq:2:KineticChar} can be recast in
	terms of $(w_t)_{t \geq 0}$ as follows:
	\begin{equation*}
		\left \{ \begin{array}{l} \partial_t w_t(\xi)- \xi\left(\alpha + \beta\frac 1{t+1}  \right) \partial_\xi w_t(\xi)
		= \widehat{\mathcal{Q}}_{(t+1)^\beta e^{\alpha t}}w_t(\xi) - w_t(\xi), \\ w_0(\xi)  = \phi(\xi). \end{array} \right.
	\end{equation*}
	Here $\widehat{\mathcal{Q}}_s$, for $s\geq 0$, denotes a rescaled smoothing transform obtained by replacing $C$ with $sC$, that is,
	\begin{equation*}
		\widehat{\mathcal{Q}}_s\psi(\xi) = \E \left[ e^{i \xi s C} \: \prod_{k=1}^N \psi(A_k\xi) \right],  \quad \xi \in \R.
	\end{equation*}	
	Suppose now that $w_t(\xi) \to w_\infty(\xi)$, as $t\to\infty$, for an appropriate choice of $\alpha$ and $\beta$ and a nondegenerate limit $w_\infty(\xi)$. Note that if $C \neq 0$ this is possible only if
	$(t+1)^\beta e^{\alpha t}\to 0$, as $t\to\infty$.
	Following~\cite{bassetti:2012:self,bobylev2009self} we can find a probabilistic interpretation of $w_\infty$.
	Using the partial differential equation for $(w_t)_{t \geq 0}$ we see that $w_\infty$  ought to satisfy
	\begin{equation}\label{eq:1:statw}
		 -\alpha \xi \: \frac{\partial}{\partial \xi} w_\infty(\xi)  = \widehat{\mathcal{Q}}_0w_\infty(\xi) - w_\infty(\xi).
	\end{equation}
	It is worth noting that this equation does not depend on the value of $\beta$. The probabilistic interpretation becomes evident after rewriting the above relation as
	\begin{equation*}
		w_\infty(\xi)  = \int_0^1 \widehat{\mathcal{Q}}_0w_{\infty}( s^{-\alpha}\xi ) \: \ud s = \E \left[\prod_{k=1}^N w_\infty(L^{-\alpha}A_k\xi) \right],
	\end{equation*}
	where $L$ is distributed uniformly on $(0,1)$ and is independent of $(N,A_1,A_2,\ldots)$. As verified by our Theorems~\ref{mthm5} and~\ref{mthm6}, $w_\infty$ is the characteristic function of a fixed point of the
	modified {\it homogeneous} smoothing transform
	\begin{equation*}
		\kappa \mapsto {\rm Law}\left(L^{-\alpha}\sum_{k=1}^N A_kX_k\right),
	\end{equation*}
	where $(X_k)_{k\in\N}$ is a sequence of independent random variables with the common distribution $\kappa$, which are also independent of the vector $(N, A_1, A_2, \ldots )$. It is worth noting that the equation~\eqref{eq:1:statw}
	is homogeneous in the argument in the sense that if $w_\infty$
	satisfies~\eqref{eq:1:statw}, then any function of the form $\xi \mapsto w_\infty(c \xi)$ also satisfies it, for $c\in\mathbb{R}$. In order to characterize $w_\infty$ it is therefore important to find an appropriate boundary condition.
	As it is verified by our main results the boundary condition for $w_\infty$, that is, the value of $\partial_\xi w_\infty(0)$ depends on the interplay between the initial condition $\phi_0$, the expectation $\EE[C]$ and the generalized
	Laplace transform $s \mapsto \EE \left[ \sum_{k=1}^N A_k^s \right]$.

	The main results of the present article show the convergence $w_t \to w_\infty$  for an appropriate choice of $\alpha$ and $\beta$. We achieve that by finding a  stochastic process
	$W=(W_t)_{t\geq 0}$ such that $\phi_t$ is the characteristic function of $W_t$, for every fixed $t\geq 0$. The dynamics of such process can be guessed from the integral form of~\eqref{eq:2:KineticChar} which reads
	\begin{equation*}
		\phi_t(\xi) = e^{-t}\phi_0(\xi) + \int_0^t e^{-s} \widehat{\mathcal{Q}}\phi_{t-s}(\xi) \ud s,\quad \xi\in\mathbb{R},\quad t\geq 0.
	\end{equation*}
	Upon plugging the definition of $\widehat{\mathcal{Q}}$ we can write, with $E$ denoting a unit mean exponential random variable independent of everything else,
	\begin{equation}\label{eq:1:prob}
		\phi_t(\xi) = \EE \left[ \ind_{\{ E >t\}} \phi_0(\xi) + \ind_{\{  E \leq t\}} e^{i \xi C} \: \prod_{k=1}^N \phi_{t-E}(A_k\xi)  \right],\quad \xi\in\mathbb{R}.
	\end{equation}
	Therefore, the initial position $W_0$ of the process $W$ must have the characteristic function $\phi_0$, and after unit mean exponential amount of time the process is replaced by a sum of $N$ independent copies of itself rescaled by
	$A_k$'s and shifted by $C$.
	Naturally, the independent copies follow the same dynamics. This in turn gives a hint that $W$ should be driven by a continuous-time branching process. We provide an in-depth construction of the process $W$ in
	Section~\ref{sec:main} very much in the spirit of~\cite{Bogus2020} where the case $C=0$ has been considered. Note, however, that even in this case our construction is new, since, unlike in~\cite{Bogus2020}, we allow $N$ to be random.
	
We finish the introduction with a brief
overview of yet another aspect of the analysis of kinetic-type equations, namely to the problem of the propagation of chaos. A system comprised of identical particles is called symmetric and chaotic with respect to some probability measure $\mu$ if, for every $k\in\mathbb{N}$, the first $k$ components of the vector of their velocities converge in distribution to $\mu^{\otimes k}$, as the number of particles tends to infinity. A peculiar result proved by Kac in~\cite{Kac:1956} is called {\it propagation of chaos} and states that if at time $0$ a system of particles exposed to elastic binary collisions is symmetric and chaotic with respect to an absolutely continuous measure $\nu_0$, then at time $t$, for every $t>0$, it is again symmetric and chaotic with respect to $\nu_t$, the solution to~\eqref{eq:1:KineticDistr}. The question of whether a particular kinetic model exhibits such a phenomenon has become classical, see, for instance~\cite{graham1997stochastic,mischler2013kac,sznitman1991topics} and
\cite{cortez2016uniform,cortez2021thermostated} for more recent results. Having mentioned that, we nevertheless shall not discuss these issues here in details, since they go beyond the scope of our paper.
	
	The article is organized as follows. In Section~\ref{sec:main} we formulate our assumptions and present the main results. In particular, we shall prove an existence and uniqueness result for the solutions to~\eqref{eq:2:KineticChar}
	by constructing a process $W$ which yields a probabilistic representation of the solution to~\eqref{eq:2:KineticChar}. The corresponding proof is given in Section~\ref{sec:ProbSol}. Many-to-one lemmas, recalled in
	Section~\ref{sec:many2one}, will be utilized for showing the convergence of the stochastic process $W$ in  Section~\ref{sec:Conv}. This in turn allows us to give short proofs of our main results in Section~\ref{sec:proof}. Throughout the paper we use the following notation. The Dirac measure at $x\in\R$ is denoted by $\delta_x$. Convergence (respectively, equality) in distribution is denoted by $\overset{d}{\to}$ (respectively, $\overset{d}{=}$). The notation $\overset{\P}{\to}$  stands for convergence in probability.

\section{Main results}\label{sec:main}

	We shall first provide some further necessary notation and formulate precisely our assumptions. Then we proceed by giving the probabilistic representation for solutions to~\eqref{eq:2:KineticChar} using a
	continuous-time branching random walk. In particular, this result yields that such solution exists and is unique. Next, we shall formulate our main results on the asymptotic behaviour of $\phi_t$, as $t\to \infty$.

\subsection{Existence and uniqueness}
	As mentioned in the introduction, $\phi_t$ solving~\eqref{eq:2:KineticChar} can be linked with a~continuous-time branching process which we shall now describe.
	In order to avoid an explosion in the branching process, and thus in $\phi_t$, we shall impose the classical non-explosion hypothesis on the distribution of $N$, see \cite[Theorem V.9.1]{harris1963theory}:
	\begin{equation}\label{eq:2:NoEx}
		\int^1_{1-\varepsilon} \frac{\ud s}{s- \EE[s^N]} = \infty \quad \mbox{for all sufficiently small } \varepsilon>0.
	\end{equation}
	Note that $\EE[N] <\infty$ is sufficient for~\eqref{eq:2:NoEx}. In what follows we also assume that, given $N$, the weights $A_1,A_2$ are a.s.~positive and are not degenerate at $1$. More precisely, we suppose that for every   $m\in\N$
	\begin{equation}\label{eq:Ai_positive}
		\P[A_1>0,A_2>0,\ldots,A_m>0|N=m]=1
\end{equation}
and for some $m$ such that $\P[N=m] >0$ it holds
\begin{equation}\label{eq:2.1'}
\P[A_1=1,A_2=1,\ldots,A_m=1|N=m]<1.
	\end{equation}
	For the sake of transparency we shall further suppose that
	\begin{equation}\label{eq:2:super_crit}
		\EE[N]\in (1,\infty],
	\end{equation}
	which is the standard supercriticality condition and which ensures that our branching process survives with a positive probability. Note that if $\EE[N] \leq 1$ then either $\PP[N=1]=1$ and the branching is degenerate, or $\PP[N=1]<1$,
	$\EE[N] \leq 1$ and the process dies out with probability one. The equation~\eqref{eq:1:KineticDistr} in both cases can be treated using our methods.

	We now turn our attention to the probabilistic interpretation of the solution to~\eqref{eq:2:KineticChar} which we shall use to establish the existence and uniqueness result. We hinge on ideas presented in~\cite{Bogus2020}.
	More precisely we shall work with a (marked) continuous-time branching random walk
	which can be described in the following fashion. At time $t=0$ one particle is placed at the origin of the real line $\R$.
	After a random amount of time, distributed according to the unit mean exponential law, the particle splits into a random number $N$ of new particles which are placed at random points of $\R$ given by a point process
	\begin{equation*}
		\zeta=\sum_{k=1}^N\delta_{Z_k}, \qquad  Z_k=\log A_k.
	\end{equation*}
	From here each particle reproduces in exactly the same way independently from other particles. In particular, the relative positions of particles with respect to their mother are distributed
	according to a copy of $\zeta$. Denote by $\mathcal{T}_\infty \subseteq \bigcup_{n\geq 0} \NN^n$, where $\NN^0:=\{\varnothing\}$, the (full) Galton-Watson tree of the underlying population with the Ulam-Harris labelling,
	that is, we label the root with $\varnothing$ and for each $x \in \mathcal{T}_\infty$ we label its children with $xi$ where $1 \leq i \leq N(x)$, with $N(x)$ denoting the number of children of $x$. For $x,y\in \mathcal{T}_\infty$, the notation $y\leq x$ means that $y$ belongs to the unique path connecting $\varnothing$ and $x$ or, equivalently, $y$ is an ancestor of $x$. To model the dynamics of the process suppose that each vertex
	$x \in \mathcal{T}_\infty$ is equipped with $(E(x), U(x), C(x), Z_1(x), Z_2(x), \ldots)$ which is an independent copy of $(E, U, C, Z_1, Z_2, \ldots)$, where $E$ denotes a unit mean exponential random variable independent of
	$(U, C, Z_1, Z_2, \ldots)$ and $U$ is a random variable independent of $(E, C, Z_1, Z_2, \ldots)$ with a law that will be specified later.
	Suppose that the triple $(E(x), U(x), C(x))$ is attached to the vertex $x \in \mathcal{T}_\infty$ and $Z_i(x)$ is attached to the edge between $x$ and $xi$.
	For every particle $x \in \mathcal{T}_\infty$ we define its time of birth $B(x)$ and the time of death $D(x)$ by the formulas
	\begin{equation*}
		B(\varnothing)=0, \qquad 		B(x) = \sum_{k=0}^{|x|-1} E(x_{|k}),\quad |x|\ge 1,\quad D(x) = \sum_{k=0}^{|x|} E(x_{|k}) = B(x)+E(x),\quad x\in\mathcal{T}_{\infty},
	\end{equation*}
 	where $x_{|k}$ is the $k$-th vertex in the unique path from the root  $\varnothing$ to $x$ and
 	$|x|$ denotes the length of this path. Note that $E(x)$ is therefore interpreted as the lifetime of particle $x \in \mathcal{T}_\infty$.
	Denote by $\Tcal_t \subseteq \mathcal{T}_\infty$ the tree discovered up to time $t>0$, given by
	\begin{equation*}
		\mathcal{T}_t = \{ x \in \mathcal{T}_\infty \: : \: B(x) \leq t \}.
	\end{equation*}
	The set of particles alive at time $t$ is given by
	 \begin{equation*}
	 	\partial \Tcal_t =  \{ x \in \mathcal{T}_\infty \: : \: B(x) \leq t  < D(x) \}.
	 \end{equation*}	
	The relative complement of $\partial \Tcal_t$,  that is
	 \begin{equation*}
	 	 \Tcal_t^o = \Tcal_t \setminus (\partial \Tcal_t) =  \{ x \in \mathcal{T}_\infty \: : \: D(x) \leq t \},
	 \end{equation*}
	 represents the particles removed from the system by time $t$. Both collections $\partial \Tcal_t$ and $ \Tcal_t^o $ will play a significant role in the forthcoming construction of the
	probabilistic representation of the solutions to~\eqref{eq:2:KineticChar}.
	A possible realization of the marked branching process described above is depicted on Figure~\ref{fig:ctbrw}.
	
	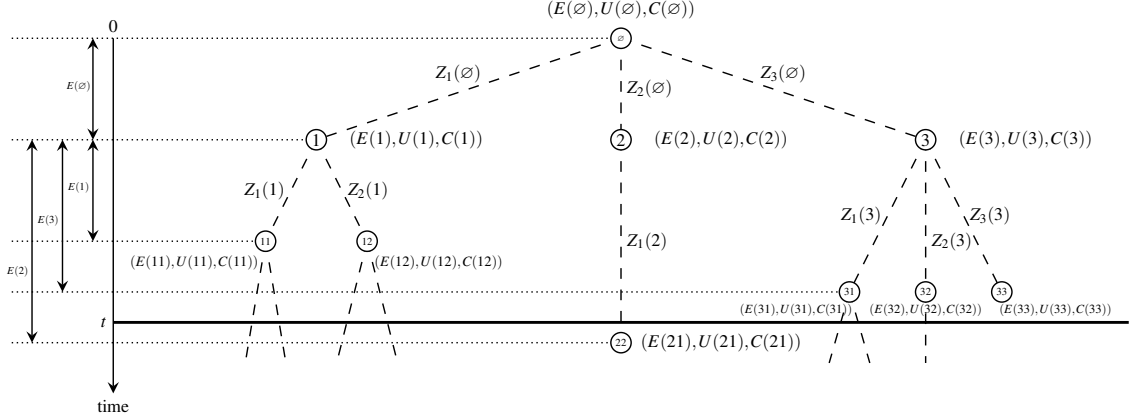
\begin{figure}[!hbtp]
\centering
 \scalebox{1.35}{\begin{tikzpicture}

\draw[dashed] (0,0) -- (-3,-1) node[xshift=28,yshift=0,scale=0.5] {$(E(1),U(1),C(1))$};

\draw[dashed] (0,0) -- (0,-1) node[xshift=28,yshift=0,scale=0.5] {$(E(2),U(2),C(2))$};

\draw[dashed] (0,0) -- (3,-1) node[xshift=28,yshift=0,scale=0.5] {$(E(3),U(3),C(3))$};

\draw[dashed] (-3,-1) -- (-3.5,-2) node[xshift=-20,yshift=-6,scale=0.4] {$(E(11),U(11),C(11))$};
\draw[dashed] (-3,-1) -- (-2.5,-2) node[xshift=20,yshift=-6,scale=0.4] {$(E(12),U(12),C(12))$};

\draw[dashed] (0,-1) -- (-0,-3) node[xshift=28,yshift=0,scale=0.5] {$(E(21),U(21),C(21))$};

\draw[dashed] (3,-1) -- (2.25,-2.5) node[xshift=-15,yshift=-5,scale=0.35] {$(E(31),U(31),C(31))$};
\draw[dashed] (3,-1) -- (3,-2.5) node[xshift=0,yshift=-5,scale=0.35] {$(E(32),U(32),C(32))$};
\draw[dashed] (3,-1) -- (3.75,-2.5) node[xshift=15,yshift=-5,scale=0.35] {$(E(33),U(33),C(33))$};

\draw[dashed] (-3.5,-2) -- (-3.7,-3.2);
\draw[dashed] (-3.5,-2) -- (-3.3,-3.2);

\draw[dashed] (-2.5,-2) -- (-2.8,-3.2);
\draw[dashed] (-2.5,-2) -- (-2.2,-3.2);

\draw[dashed] (2.25,-2.5) -- (2.05,-3.2);
\draw[dashed] (2.25,-2.5) -- (2.45,-3.2);

\draw[dashed] (3.0,-2.5) -- (3.0,-3.2);

\draw[densely dotted] (-6,0) -- (-0,0);
\draw[densely dotted] (-6,-1) -- (-3,-1);
\draw[densely dotted] (-6,-2) -- (-3.5,-2);
\draw[densely dotted] (-6,-2.5) -- (2.25,-2.5);
\draw[densely dotted] (-6,-3) -- (0,-3);

\draw[fill=white] (0,0) circle (0.1) node[scale=0.25] {$\varnothing$} node[xshift=0,yshift=8,scale=0.5] {$ (E(\varnothing),U(\varnothing),C(\varnothing))$};
\draw[fill=white] (-3,-1) circle (0.1) node[scale=0.5] {$1$};
\draw[fill=white] (0,-1) circle (0.1) node[scale=0.5] {$2$};
\draw[fill=white] (3,-1) circle (0.1) node[scale=0.5] {$3$};
\draw[fill=white] (-3.5,-2) circle (0.1) node[scale=0.3] {$11$};
\draw[fill=white] (-2.5,-2) circle (0.1) node[scale=0.3] {$12$};
\draw[fill=white] (-0,-3) circle (0.1) node[scale=0.3] {$22$};
\draw[fill=white] (2.25,-2.5) circle (0.1) node[scale=0.3] {$31$};
\draw[fill=white] (3,-2.5) circle (0.1) node[scale=0.3] {$32$};
\draw[fill=white] (3.75,-2.5) circle (0.1) node[scale=0.3] {$33$};

\node[above,scale=0.5] (c) at (-1.6,-0.5) {$Z_1(\varnothing)$};
\node[right,scale=0.5] (c) at (0,-0.5) {$Z_2(\varnothing)$};
\node[above,scale=0.5] (c) at (1.6,-0.5) {$Z_3(\varnothing)$};

\node[left,scale=0.5] (c) at (-3.25,-1.5) {$Z_{1}(1)$};
\node[right,scale=0.5] (c) at (-2.75,-1.5) {$Z_{2}(1)$};

\node[right,scale=0.5] (c) at (0,-2.0) {$Z_1(2)$};

\node[left,scale=0.5] (c) at (2.625,-1.75) {$Z_{1}(3)$};
\node[right,scale=0.5] (c) at (3.0,-2.0) {$Z_{2}(3)$};
\node[right,scale=0.5] (c) at (3.375,-1.75) {$Z_{3}(3)$};

\draw[thick,black] (-5,-2.8) -- (5,-2.8);
\draw[-stealth] (-5,0) -> (-5,-3.5);

\draw[stealth-stealth] (-5.2,0) -> (-5.2,-1.0) node[xshift=-4,yshift=15,scale=0.3] {$E(\varnothing)$};

\draw[stealth-stealth] (-5.2,-1.0) -> (-5.2,-2.0) node[xshift=-4,yshift=15,scale=0.3] {$E(1)$};

\draw[stealth-stealth] (-5.5,-1.0) -> (-5.5,-2.5) node[xshift=-4,yshift=20,scale=0.3] {$E(3)$};

\draw[stealth-stealth] (-5.8,-1.0) -> (-5.8,-3.0) node[xshift=-4,yshift=20,scale=0.3] {$E(2)$};

\node[above,scale=0.5] (c) at (-5.0,0) {$0$};

\node[left,scale=0.5] (c) at (-5.0,-2.8) {$t$};

\node[below,scale=0.5] (c) at (-5.0,-3.5) {$\text{time}$};

\end{tikzpicture}}
\caption{A realization of a marked continuous-time branching random walk. The horizontal thick line represents a fixed time $t>0$. Note that no dashed lines emanating from the vertex $33$ downward means that $N(33)=0$. On the figure it is also assumed that $E(\varnothing)+E(3)+E(33)>t$, that is, the particle $33$ dies after time $t$. The tree discovered up to time $t$ is $\mathcal{T}_t=\{\varnothing,1,2,3,11,12,31,32,33\}$ and consists of all vertices above the thick line. The set of particles alive at time $t$ is $\partial\mathcal{T}_t=\{2,11,12,31,32,33\}$ and consists of vertices above the thick line which have not split above this line. The set of particles
removed from the system by time $t$ is $\mathcal{T}_t^{\circ}=\{\varnothing,1,3\}$.}
\label{fig:ctbrw}
\end{figure}

	A continuous-time branching random walk is a branching process with a spatial component. The position $V(x)$ of
	$x \in \mathcal{T}_\infty$ is obtained by summing the edge weights along the unique path from the root $\varnothing$ to~$x$. Thus, for $x \in \mathcal{T}_\infty$,
	\begin{equation*}
		V(x) = \sum_{k=1}^{|x|} Z_{x_{|k}}(x_{|k-1}).
	\end{equation*}
	Therefore, the positions of particles present in the system at time $t\geq 0$ are given by the point process
	\begin{equation*}
		\Zcal_t = \sum_{ x \in \partial\Tcal_t}\delta_{V(x)} .
	\end{equation*}
	The measure-valued process $\Zcal = (\Zcal_t)_{t \geq 0}$ is called a continuous time branching random walk.
	
	Note that the variables $(U(x),C(x))_{x\in\mathcal{T}_{\infty}}$ did not play a role so far in the construction of $\Zcal$. We shall now make use of them for two markings. The first one is given via an independent marking of
	$\Zcal$ with independent copies of $U$, namely
	\begin{equation*}\label{eq:3:calXt}
		\Xcal_{U,t} = \sum_{ x \in \partial\Tcal_t}\delta_{V(x)} \otimes \delta_{U(x)},\quad t\geq 0.
	\end{equation*}
The point process $(\Xcal_{U,t})_{t\ge 0}$ will be mainly used to control the contribution of the initial value by taking $U=R$, where, recall, $R$ is a generic random variable distributed according to $\nu_0$, but it will be also applied in a slightly different aspect
	(with a differently distributed $U$) as well.
	
 	The second component needed  to construct $\phi_t$ controls the contribution of the inhomogeneous terms (shifts) $C(x)$. Put
	\begin{equation*}\label{eq:3:calCt}
		\Ccal_t = \sum_{ x \in \Tcal_t^o}\delta_{ V(x) } \otimes \delta_{C(x)},\quad t\geq 0.
	\end{equation*}
	Finally, define $\Wcal = (\Wcal_t)_{t \geq 0}$ via
	\begin{equation}\label{eq:W_cal_definition}
		\Wcal_t = \Xcal_{R,t}+\Ccal_t,\quad t\geq 0,
	\end{equation}
	where $R$ is distributed according to $\nu_0$.
	
	Our first main result provides an explicit form of a unique solution to~\eqref{eq:2:KineticChar} which turns out to be the characteristic function of a certain integral with respect to the point process $\Wcal_t$.

	\begin{thm}\label{mthm1}
		Assume~\eqref{eq:2:NoEx},~\eqref{eq:Ai_positive},~\eqref{eq:2.1'} and~\eqref{eq:2:super_crit}. Then in the class~\eqref{eq:class_of_functions} there exists a unique solution to \eqref{eq:2:KineticChar} given by $\phi_t(\xi) = \EE \exp\{ i\xi W_t\}$, where
		\begin{equation*}
			W_t =  \int e^{v}z \:\Wcal_t(\ud v, \ud z)
		\end{equation*}
		and $\Wcal_t$ is defined by~\eqref{eq:W_cal_definition}.
	\end{thm}	
	
	The proof of Theorem~\ref{mthm1} given in Section~\ref{sec:ProbSol}  utilizes~\eqref{eq:1:prob} and the branching properties of $\Xcal_{U,t}$ and $\Ccal_t$ which will also be discussed in Section~\ref{sec:ProbSol}.
	
\begin{rem}
Probabilistic interpretations of the solutions to~\eqref{eq:1:KineticDistr} on various levels of generality have a long history. For the original Kac model a probabilistic interpretation was given by McKean in~\cite{mckean1966speed}. A probabilistic representation for the inelastic Kac model is contained in~\cite{bassetti2008probabilistic} and further refinements can be found in~\cite{bassetti:2012:self,bassetti2011central,Bogus2020}. In the context of propagation of chaos other types of probabilistic representation have been proposed in~\cite{cortez2016uniform,cortez2016quantitative}. A particular case of Theorem~\ref{mthm1} with $N=2$ can be extracted from Proposition 4 in~\cite{Bassetti2011}. In particular, the decomposition $W_t=X_{R,t}+C_t$ is very similar in spirit to the decomposition right before Proposition 4 in~\cite{Bassetti2011}. A detailed comparison of the approaches to probabilistic representations of solutions to~\eqref{eq:1:KineticDistr} via random labelled $N$-ary recursive trees (proposed in~\cite{bassetti:2012:self,Bassetti2011}) and via continuous-time branching random walks (proposed in~\cite{Bogus2020}) can be found in Remark 2.7 in~\cite{Bogus2020}.
\end{rem}

\subsection{Asymptotics}

	We shall now explain how one can use the asymptotic behaviour of $\Wcal_t$ to show that $\phi_t$, appropriately scaled, possesses a non-trivial limit as $t \to\infty$.
	It turns out that the underlying branching process, by a change of measure
	argument, is related to a compound Poisson process. The asymptotic of the latter is described in terms of the function $\Phi:[0,\infty)\to [-1,+\infty]$ defined by
	\begin{equation*}
		\Phi(s) = \EE \left[\sum_{k=1}^N A_k^s \right]-1,\quad s\geq 0.
	\end{equation*}
	Let $\mathcal{D} =\left\{t>0 \: : \:\EE \left[\sum_{k=1}^N A_k^t \right]<\infty \right\}$ and assume that $\mathcal{D}$ is not empty. Note that $\Phi(0)=\E[N]-1$ and, particularly, $\E[N]<\infty$ iff $\Phi(0)<\infty$. Moreover, by the merit of H\"{o}lder's inequality, $\mathcal{D}$ is a convex set. The function $\Phi$ is continuous and strictly convex on the open interval ${\rm int}( \mathcal{D})$.
	
	The fundamental object in our context is the so-called additive martingale
	\begin{equation}\label{eq:star}
		M_t(\gamma) =e^{-\Phi(\gamma) t}\int e^{\gamma v}\Zcal_t(\ud v), \qquad \gamma\in \mathcal{D}
	\end{equation}
	If $\Phi(\gamma) > \gamma \Phi'(\gamma)$, then under some mild moment assumptions, see Lemma~\ref{lem:5:MartConvSub} below, $M_t(\gamma)$ converges almost surely and in $L_1$ to a random variable
	$M_\infty(\gamma)$ which is
	almost surely positive on the set $\{|\mathcal{T}_\infty| = \infty \}$. If, on the other hand, $\Phi(\gamma) \leq \gamma \Phi'(\gamma)$ then $M_t(\gamma)$ still has a nontrivial limit in a weaker sense (in probability or distribution)
	after a polynomial rescaling. The martingale limits $M_\infty(\gamma)$ will appear naturally in our arguments.
	
As $\Wcal_t = \Xcal_{R,t}+\Ccal_t$ it is convenient to express $W_t$ as a sum of two terms $X_{R,t}$ and $C_t$ defined by the following identities. Put
	\begin{equation}\label{eq:3:Xt}
		X_{U,t}:=\int e^{v}u \:\Xcal_{U,t}(\ud v,\ud u), \quad t \geq 0,
	\end{equation}
	where $U$ denotes some random variable (we use here $U$ instead of $R$, because this notation will be applied below in various contexts) and
		\begin{equation}\label{eq:3:Ct}
		C_t:=\int e^{ v}c \:\Ccal_t(\ud v,\ud c), \quad t\geq 0.
	\end{equation}
	In our settings we shall scale this processes using the so-called spectral function given by
	\begin{equation*}
		\mu(s) = \Phi(s)/s,\quad s \in \mathcal{D}.
	\end{equation*}
	Since $\Phi$ is strictly convex, the spectral function $\mu$ can possess at most one minimizer $\gamma^*$. If it exists, it is then the unique solution to the equation $\mu'(\gamma^*)=0$, or equivalently
	${\Phi(\gamma^*)} = {\gamma^*} \Phi'(\gamma^*)$. We shall mainly consider values of $\gamma$ smaller than $\gamma^*$. However, since we do not assume existence of $\gamma^*$, instead of referring to the inequality
	$\gamma<\gamma^{\ast}$, we shall use the condition $\mu'(\gamma)<0$, or equivalently $\Phi(\gamma) > \gamma \Phi'(\gamma)$. The latter condition will allow us to control the dynamics of $X_{R,t}$ provided (roughly speaking) that the distribution of $R$ lies in the domain of normal attraction of a stable law of index $\gamma \in (0,2]$.
	More specifically (and similarly to \cite{Bassetti2011,Bogus2020}) in some cases we shall work under the following assumptions imposed on the initial condition $\nu_0$.
	We say that condition ($H_{\gamma}$) is satisfied if:
	\begin{itemize}
		\item for $\gamma = 1$ one of the following conditions (a) or (b) is fulfilled:
			\begin{itemize}
				\item[(a)] $\int |x| \nu_0(\ud x) <\infty$ and in this case we put $m_0 = \int x \nu_0(\ud x)$;
				\item[(b)] we have
					\begin{equation*}
						\lim_{x \to +\infty} x \nu_0(x, +\infty) = \lim_{x \to -\infty} |x| \nu_0(-\infty, x) = c_0 \in (0, \infty),
					\end{equation*}
					and
					\begin{equation*}
						\lim_{T \to +\infty} \int_{[-T,T]} x \nu_0(\ud x) = m_0 \in (-\infty, \infty);
					\end{equation*}
			\end{itemize}
		\item for $\gamma=2$,  $\sigma_0^2 = \int x^2\nu_0(\ud x) <\infty$ and $\int x\nu_0(\ud x)=0$;
		\item for $\gamma\in (0,1)\cup(1,2)$ it holds
			\begin{equation*}
				\lim_{x \to +\infty}x^\gamma \nu_0(x, +\infty) = c_0^{+}, \quad \lim_{x \to -\infty} |x|^\gamma \nu_0(-\infty, x) = c_0^{-}
			\end{equation*}
			with $c_0^++c_0^->0$. If $\gamma>1$, then furthermore $\int x \nu_0(\ud x)=0$.
	\end{itemize}
	The process $(X_{R,t})_{t\geq 0}$ is related to the homogeneous kinetic equation. More precisely, consider the equation
	\begin{equation}\label{eq:hom_kinetic}
		\left\{
		\begin{array}{l}
		\partial_t \tilde\phi_t = \tilde \Qcal \tilde \phi_t - \tilde \phi_t,\\
		\tilde\phi_0(\xi) =\phi_0(\xi) = \EE[\exp\{i\xi R\}];
		\end{array}
		\right.
	\end{equation}	
	where $\tilde \Qcal$ is the homogeneous smoothing transform obtained from~\eqref{eq:inhom_smooth_def} by putting $C:=0$, that is,
$$
\tilde\Qcal\phi(\xi) = \E \left[ \prod_{k=1}^N \phi(A_k\xi) \right].$$
Under mild hypotheses, see Proposition 2.5 in~\cite{Bogus2020}, $\tilde \phi_t(\xi) = \EE[\exp\{ i \xi X_{R,t}\}]$ is a unique solution to~\eqref{eq:hom_kinetic}
	and its asymptotic depends on the sign of $\mu'(\gamma)$ (or in other words on the mutual position of $\gamma$ and $\gamma^*$, if $\gamma^*$ exists).
	We state below results proved\footnote{Note that in \cite{bassetti:2012:self,Bogus2020} the authors assume that $N$ is a constant, however this hypothesis is superfluous and can be easily replaced by~\eqref{eq:2:NoEx}
	and~\eqref{eq:2:super_crit}.} in \cite{bassetti:2012:self,Bogus2020,Buraczewski2019}.
	Assume that condition ($H_{\gamma}$) holds.
 	\begin{itemize}
   		\item If $\mu'(\gamma) < 0$, then
			\begin{equation}\label{eq:a1}
				e^{-\mu(\gamma)t}X_{R,t} \overset{d}{\to} X_{\infty}, \qquad t\to\infty.
			\end{equation}
		\item If $\mu'(\gamma) = 0$, then
			\begin{equation}\label{eq:a2}
				t^{\frac 1{2\gamma}}e^{-\mu(\gamma)t}X_{R,t} \overset{d}{\to} X_{\infty}, \qquad t\to\infty.
			\end{equation}
		\item If $\mu'(\gamma)>0$, and $\gamma^{\ast}$ exists, then
 			\begin{equation}\label{eq:a32}
				t^{\frac 3{2\gamma^*}}e^{-\mu(\gamma^*)t}X_{R,t} \overset{d}{\to} X_{\infty}, \qquad t\to\infty.
			\end{equation}
 	\end{itemize}
	Let us emphasize here that in the sequel we shall use only~\eqref{eq:a1}, but we present here all possible results for $X_{R,t}$ to provide the reader with a general overview. In all the cases the limit satisfies the following
	stochastic fixed-point equation
	$$
		X_\infty \overset{d}{=} L^{\mu(\gamma\wedge \gamma^*)}\sum_{k=1}^N A_k X_{k,\infty},
	$$
	where $(X_{k,\infty})_{k\in \N}$ are independent copies of $X_\infty$, $L$ is uniformly distributed on $(0,1)$
	and $(N,A_1,A_2,\ldots)$, $L$, $(X_{k,\infty})_{k\in \N}$ are mutually independent. All the above results can be stated in terms of the characteristic functions and since the most relevant for us will be the regime when~\eqref{eq:a1}
	holds, we provide further details only in this case. Let
	\begin{equation*}
		\hat{g}_\gamma (\xi) =
		\begin{cases}
		e^{im_0\xi}, & \mbox{if $\gamma=1$ and $(H_1(a))$ holds}, \\
			e^{i m_0 \xi-\pi c_0|\xi|}, & \mbox{if $\gamma=1$ and $(H_1(b))$ holds},\\
				e^{-\sigma_0^2 \xi^2/2}, &   \mbox{if $\gamma=2$ and $(H_2)$ holds}, \\
				e^{-k_0 |\xi|^\gamma(1-ib_0 \tan(\pi \gamma/2)){\rm sgn}(\xi) }, &  \mbox{if $\gamma\in (0,1) \cup (1,2)$ and $(H_\gamma)$ holds};
		\end{cases}
	\end{equation*}
	where
	$$
	k_0=(c_0^{+}+c_0^{-})\frac{\pi}{2\Gamma(\gamma)\sin(\pi\gamma/2)},\quad  b_0=\frac{c_0^{+}-c_0^{-}}{c_0^{+}+c_0^{-}}.
	$$
	Then convergence \eqref{eq:a1} is equivalent to
	\begin{equation*}
  		\tilde\phi_t(e^{-\mu(\gamma)t}\xi)=\E \big[ e^{i\xi e^{-\mu(\gamma)t} X_t} \big] \to \tilde \phi_\infty(\xi)=\E\big[ e^{i\xi X_{\infty}} \big], \qquad t\to\infty,
	\end{equation*}
	where the limiting function has a representation in terms of the martingale limit
	\begin{equation}\label{eq_x_infinity_cf}
		\tilde{\phi}_\infty(\xi)=\E\big[ e^{i\xi X_{\infty}} \big] = \E\big[\hat{g}_\gamma   (\xi M_{\infty}^{1/\gamma}(\gamma)) \big],\quad\xi\in\R,
	\end{equation}
	see Theorem 2.2 in~\cite{bassetti:2012:self}.

	Taking into account the above results for $X_{R,t}$, to describe the behaviour of $W_t=X_{R,t}+C_t$ for large $t$, it is necessary to understand the asymptotic behaviour of $C_t$ defined in~\eqref{eq:3:Ct}. We summarize the results in the following two theorems. The
	first one treats the case when $C$ has a non-zero mean.
	In this case we assume
	\begin{equation}\label{eq:lp1}
		\E [C] \not= 0,\ \E |C|^p<\infty \mbox{ for some $p>1$},\  1\in {\rm int}(\mathcal{D}) \mbox{ and } \mu'(1)<0.
	\end{equation}
 	We shall consider three subcases:
	\begin{itemize}
		\item [(A)] $\Phi(1) >0$ and \eqref{eq:lp1} holds;
		\item [(B)] $\Phi(1) =0$ and \eqref{eq:lp1} holds;
		\item [(C)] $\Phi(1) < 0 $ and \eqref{eq:lp1} holds.
	\end{itemize}

	\begin{thm}\label{mthm2}
 		Assume~\eqref{eq:2:NoEx},~\eqref{eq:Ai_positive},~\eqref{eq:2.1'} and~\eqref{eq:2:super_crit}. The following limit relations hold true:
		\begin{itemize}
		\item[$(a)$] under condition $(A)$,
			\begin{equation*}
				e^{-\mu(1)t}C_t  \overset{\P}{\to} \frac{\EE[C]}{\Phi(1)} M_\infty(1),\quad t\to\infty;
			\end{equation*}
		\item[$(b)$] under condition $(B)$,
			\begin{equation*}
				\frac 1t C_t \overset{\P}{\to} \EE[C] M_\infty(1),\quad t\to\infty;
			\end{equation*}	
		\item[$(c)$] under condition $(C)$,
				\begin{equation*}
					 C_t\overset{\P}{\to} C_\infty:= \sum_{x \in \mathcal{T}_\infty} e^{V(x)}C(x),\qquad t\to\infty,
				\end{equation*}	
			where $C_{\infty}$ is a.s.~finite and satisfies the equality in distribution
			\begin{equation}\label{eq:inhomogeneous_smoothing_transform}
				C_\infty  \stackrel{d}{=} \sum_{k=1}^NA_k C_{k,\infty} +C,
			\end{equation}
			with $(C_{k, \infty})_{k\in\N}$ being independent copies of $C_\infty$ also independent of $(C, N, A_1, A_2,\ldots)$.	
	\end{itemize}	
	\end{thm}

	To state the results in the zero-mean case we need the following sets of conditions:
	\begin{itemize}
		\item [(D)] $\EE[C]=0$, $\E[|C|^p]<\infty$ and $\Phi(p)<0$ for some $p\in (1,2]$, and $\mu'(1) < 0$;
		\item [(E)] $\EE[C]=0$, $2\in {\rm int}(\mathcal{D})$, $\Phi(2)>0$, $\mu'(2)<0 $  and $\E [C^2] \in (0,\infty)$.
	\end{itemize}

	\begin{thm}\label{mthm3}
		Suppose that~\eqref{eq:2:NoEx},~\eqref{eq:Ai_positive},~\eqref{eq:2.1'} and~\eqref{eq:2:super_crit} are satisfied. Then,
		\begin{itemize}
			\item[$(d)$] under condition $(D)$,
				\begin{equation*}
					C_t \overset{\P}{\to} C_{\infty}, \qquad t\to\infty.
				\end{equation*}	
			\item[$(e)$]  under condition $(E)$,
				$$
					e^{-\mu(2)t} C_t  \overset{d}{\to} \sqrt{M_\infty(2)(\E [C^2]) /\Phi(2)}\; {\rm N},\quad t\to\infty,
				$$
				where ${\rm N}$ is a standard normal random variable  independent of the martingale limit $M_{\infty}(2)$.
		\end{itemize}
	\end{thm}

	\begin{rem}
		The inhomogeneous stochastic fixed-point equation of the form~\eqref{eq:inhomogeneous_smoothing_transform} has received considerable attention in the past decades. For example, the set of solutions has been
		characterized in~\cite{Alsmeyer+Meiners:2012} (case of nonnegative $C$) and  \cite{Alsmeyer+Meiners:2013} (two-sided case). In particular, under the assumptions (C) or (D),
		Theorem 2.3 in \cite{Alsmeyer+Meiners:2013} implies that the set of solutions is indeed nonempty and provides a description of all the solutions. The variable $C_{\infty}$ appearing in
		Theorems~\ref{mthm2} and~\ref{mthm3} is a particular solution.
	\end{rem}

	Theorems~\ref{mthm2} and~\ref{mthm3} will be proved in Section \ref{sec:Conv} with the only exception being the case (E) in Theorem \ref{mthm3}. The latter will follow from case (E) in
	Theorem~\ref{mthm4} by taking $R=0$, that is $\nu_0 = \delta_0$.

	In order to formulate our results concerning asymptotic behaviour of $W_t$ one needs to compare the influence of $C_t$ (Theorems \ref{mthm2} and \ref{mthm3}) with $X_{R,t}$ (formulas \eqref{eq:a1}~--~\eqref{eq:a32}).
	Observe that our results concerning the asymptotic behaviour of $C_t$ are stated either under the assumption $\mu'(1)<0$ or under the assumption $\mu'(2)<0$. In both cases $\mu(\gamma^*)$ (assuming that
	$\gamma^{\ast}$ exists) is strictly smaller than $\mu(1)$ or $\mu(2)$. Thus, in the regimes \eqref{eq:a2} and \eqref{eq:a32} $C_t$ always grows faster than $X_{R,t}$. Furthermore, this still holds true if we replace the condition
	$(H_{\gamma})$ (needed, for \eqref{eq:a2} and \eqref{eq:a32}) with an appropriate moment assumption on $\nu_0$. Summarizing, we shall only consider below $X_{R,t}$ either in the regime~\eqref{eq:a1}, that is,
	when $\mu'(\gamma)<0$ (Theorems \ref{mthm5} and \ref{mthm6}) or under some moment assumptions on $\nu_0$ ensuring that $C_t$ grows faster than $X_{R,t}$ (Theorem \ref{mthm6}).

	Recall that Theorem~\ref{mthm1} implies $\phi_t(\xi) = \E \exp\{i\xi W_t\} = \E \exp\{i\xi (X_{R,t}+C_t)\}$.
	We start with the case when influence of both ingredients $C_t$
	and $X_{R,t}$ is comparable.  The following result will be proved in Section~\ref{sec:proof}.

	\begin{thm}\label{mthm4}
		Assume that~\eqref{eq:2:NoEx},~\eqref{eq:Ai_positive},~\eqref{eq:2.1'} and~\eqref{eq:2:super_crit} are satisfied.
		\begin{itemize}
			\item If conditions $(A)$ and $(H_1)$ are satisfied, then
				$$
					\phi_t(e^{-\mu(1)t }\xi) = \E \big[ e^{i\xi e^{-\mu(1)t } W_t} \big] \to   \E\big[\hat{g}_1 (\xi M_{\infty}(1)) e^{i\xi \E [C]/\Phi(1) \cdot M_\infty(1)} \big], \qquad t\to \infty.
				$$
			\item If conditions $(C)$ and $(H_\gamma)$ are satisfied for some $\gamma <1$ such that $\Phi(\gamma)=0$, then
				$$
					\phi_t(\xi) = \E \big[ e^{i\xi W_t} \big] \to \E\big[\hat{g}_\gamma   (\xi M_{\infty}^{1/\gamma}(\gamma)) e^{i\xi C_\infty} \big],\quad t\to\infty,
				$$
			\item If conditions $(D)$ and $(H_\gamma)$ are satisfied for some $\gamma <p$ such that $\Phi(\gamma)=0$, then
				$$
					\phi_t(\xi) = \E \big[ e^{i\xi W_t} \big] \to \E\big[\hat{g}_\gamma   (\xi M_{\infty}^{1/\gamma}(\gamma)) e^{i\xi C_\infty} \big],\quad t\to\infty,
				$$
			\item If conditions $(E)$ and $(H_2)$ are satisfied, then
				$$
					\phi_t(e^{-\mu(2)t}\xi)=\E [e^{i\xi e^{-\mu(2)t}W_t}] \to \E \exp\left\{-\xi^2\left(M_{\infty}(2)(\sigma_0^2+\E [C^2]/\Phi(2)\right)/2\right\},\quad t\to\infty.
				$$
		\end{itemize}
	\end{thm}

	\begin{rem}
		Note that if condition $(B)$ holds, then $C_t/t$ converges in probability and there is no regime in which $X_{R,t}$ also grows linearly meaning that $C_t$ and $X_{R,t}$ cannot be of the same magnitude.
		This explains why the case $(B)$ is excluded in Theorem~\ref{mthm4}.
	\end{rem}

	Now we consider the case when $X_{R,t}$ dominates over $C_t$ and thus determines asymptotic behaviour of $W_t$. Our next result is a direct consequence of the  Slutsky theorem.

	\begin{thm}\label{mthm5}
		Assume~\eqref{eq:2:NoEx},~\eqref{eq:Ai_positive},~\eqref{eq:2.1'} and~\eqref{eq:2:super_crit}. Suppose further that  one of the following set of conditions holds:
		\begin{itemize}
			\item $(A)$ and $(H_\gamma)$  for $\gamma <1$;
			\item $(B)$ and $(H_\gamma)$  for $\gamma <1$;
			\item $(C)$ and $(H_\gamma)$ for $\gamma <1$ and $\Phi(\gamma)>0$;
			\item $(D)$ and $(H_\gamma)$ for $\gamma <p$ and $\Phi(\gamma) >0$;
			\item $(E)$ and $(H_\gamma)$  for $\gamma <2$.
 		\end{itemize}
		Then
		$$
			\phi_t(\xi e^{-\mu(\gamma)t})\to \E\exp(i\xi X_\infty),\qquad t\to\infty,
		$$
		where the characteristic function on the right-hand side is given by~\eqref{eq_x_infinity_cf}.
	\end{thm}

	Finally, when $C_t$ dominates, we do not need to refer to conditions $(H_{\gamma})$ and, instead, finiteness of certain moments of $R$ is sufficient.

	\begin{thm}\label{mthm6}
		Assume that~\eqref{eq:2:NoEx},~\eqref{eq:Ai_positive},~\eqref{eq:2.1'} and~\eqref{eq:2:super_crit} are satisfied. Then,
		\begin{itemize}
			\item If (A) holds and, further, $\E [R] = 0$ and $\E [|R|^{1+\delta}]<\infty$ for some $\delta >0$, then
				\begin{equation*}
					e^{-\mu(1)t}W_t  \overset{\P}{\to} \frac{\EE[C]}{\Phi(1)} M_\infty(1)\quad\text{and}\quad\phi_t(e^{-\mu(1)t}\xi)\to \E\exp(i\xi (\EE [C])	M_{\infty}(1)/\Phi(1)),\quad t\to\infty.
				\end{equation*}
			\item If (B) holds and, further, either condition $(H_1(b))$ holds, or $\E [R] = 0$ and $\E [|R|^{1+\delta}]<\infty$ for some $\delta >0$, then
				\begin{equation*}
					\frac 1t W_t \overset{\P}{\to} \EE[C] M_\infty(1)\quad\text{and}\quad\phi_t(t^{-1}\xi)\to \E\exp(i\xi (\EE [C])	M_{\infty}(1)),\quad t\to\infty.
				\end{equation*}	
			\item If (C) holds and, further, there exists $\gamma<1$ such that $\Phi(\gamma)<0$ and $\E[|R|^\gamma]<\infty$, then
				\begin{equation*}
			 		W_t\overset{\P}{\to} C_{\infty}\quad\text{and}\quad\phi_t(\xi)\to \E\exp(i\xi C_{\infty}),\quad t\to\infty.
				\end{equation*}	
			\item If (D) holds and, further, there exists $\gamma < p$ such that $\Phi(\gamma)< 0$, $\E |R|^\gamma <\infty$ and $\E [R] = 0$, if $\gamma > 1$, then
				\begin{equation*}
					W_t \overset{\P}{\to} C_{\infty}\quad\text{and}\quad\phi_t(\xi)\to \E\exp(i\xi C_{\infty}),\quad t\to\infty.
			\end{equation*}	
		\end{itemize}
	\end{thm}

	Note that there is no case $(E)$ in Theorem~\ref{mthm6} since it is essentially contained in the fourth part of Theorem~\ref{mthm4}. We provide a short proof of Theorem \ref{mthm6} in Section \ref{sec:proof}. Theorems~\ref{mthm4},\ref{mthm5} and \ref{mthm6} are complementary to the results proved in~\cite[Proposition 3]{Bassetti2011},
where the authors assumed that $N=2$ a.s. and were interested in the convergence of $\nu_t$ to a fixed point of an inhomogeneous
smoothing transform, without any additional scaling. Those results correspond to our cases (C) and (D). Hypothesis ($H_2$) in \cite{Bassetti2011} entails our condition (C), however, when $\E[C] = 0$, assumption ($H_3$) in \cite{Bassetti2011} and condition (D) are of different nature.
Finally, notice that our results do not provide the rate of convergence as the results of~\cite{Bassetti2011}.

\section{Utilizing the branching property: proof of Theorem \ref{mthm1}}\label{sec:ProbSol}
	
	In this section we shall demonstrate how one can use the branching property of continuous-time branching random walks to prove the existence and uniqueness of solutions to~\eqref{eq:2:KineticChar}.
	The key fact that we are going to use throughout this section is that given $\mathcal{T}_t$ the collections
$((E(x), U(x), C(x), Z_1(x), Z_2(x), \ldots))_{{x \in \mathcal{T}_\infty}, x\geq y}$ for $y \in \partial\mathcal{T}_t$ are conditionally i.i.d.

	The process $Y_t = |\partial\mathcal{T}_t|$ which counts the number of particles in $\partial\mathcal{T}_t$ is usually called a continuous-time branching process or a Yule process, see~\cite[Chapter III]{athreya:1972:branching}.
	  Note that, according to~\eqref{eq:2:NoEx}, $Y_t$ is finite a.s. for each $t\geq 0$. Moreover, by the memoryless property of the exponential distribution, $Y=(Y_t)_{t \geq 0}$ is a Markov process with the jump rates
	 \begin{equation}\label{eq:3:jumpsY}
	 	\P \left[ Y(t+\ud t) =k \: | \: Y(t)=j \right] = j\P[N=k-j+1] \ud t + o(\ud t)
	 \end{equation}
	 for $k \geq j-1$.

	We shall now state the branching property for $(\mathcal{X}_{U,t})_{t\geq 0}$ and $(\mathcal{C}_t)_{t \geq 0}$ that will be utilized in the proof of Theorem~\ref{mthm1}. First, note that for $t,s >0$,
	$\Xcal_{U,t+s}$ can be written in the following fashion
	\begin{equation*}
		\Xcal_{U,t+s} =   \sum_{y \in \partial \Tcal_s}\sum_{ x \in \partial\Tcal_{t+s}, y \leq x}\delta_{V(x) - V(y) + V(y)} \otimes \delta_{U(x)}.
	\end{equation*}
	Thus, if we denote for $y\in \partial \Tcal_s$
	\begin{equation*}
		\Xcal^{(y)}_{U,s, s+t}=\sum_{ x \in \partial\Tcal_{t+s}, y \leq x}\delta_{V(x) - V(y)} \otimes \delta_{U(x)},
	\end{equation*}	
	then the previous relation takes the form
	\begin{equation}\label{eq:3:branchingX}
		\Xcal_{U,t+s}(\cdot, \cdot) =  \sum_{y \in \partial \Tcal_s} \Xcal^{(y)}_{U,s, s+t}( \cdot - V(y), \cdot).
	\end{equation}
	The key feature of this representation is the following property. For every fixed $s\geq 0$, given $\partial \mathcal{T}_s$, the point processes $(\Xcal^{(y)}_{U,s, s+t}( \cdot , \cdot))_{t\geq 0}$, $y\in\partial \mathcal{T}_s$ are independent copies of $\Xcal_{U,t}$. In other words, $\Xcal_{U,t+s}$ is a sum of $|\partial \mathcal{T}_s|$ (conditionally) independent copies of $\Xcal_{U,t}$ appropriately shifted.
	The process $\Ccal= (\Ccal_t)_{t \geq 0}$ also enjoys a version of the branching property. For $s,t\geq 0$ it holds
	\begin{align}
		\Ccal_{t+s}   = \Ccal_s(\cdot, \cdot) + \sum_{y \in \partial \Tcal_s} \Ccal^{(y)}_{s, s+t}( \cdot - V(y), \cdot) \label{eq:3:branchingC}
	\end{align}
	with
	\begin{equation*}
		\Ccal^{(y)}_{s, s+t}=\sum_{ x \in \partial\Tcal_{t+s}, y \leq x}\delta_{V(x) - V(y)} \otimes \delta_{C(x)}, \qquad y\in \partial \Tcal_s.
	\end{equation*}	
	Once again, for every fixed $s\geq 0$, given $\partial \mathcal{T}_s$,
$(\Ccal^{(y)}_{s, s+t})_{y \in \partial \Tcal_s}$  are independent copies of $\Ccal_t$.

	Using~\eqref{eq:3:branchingX} and~\eqref{eq:3:branchingC} we can write, with $\Wcal^{(y)}_{s, s+t} = \Xcal^{(y)}_{R,s, s+t} + \Ccal^{(y)}_{s, s+t}$,
	\begin{equation}\label{eq:3:decW}
		\Wcal_{t+s}(\cdot, \cdot) =  \Ccal_s(\cdot, \cdot) + \sum_{y \in \partial \Tcal_s} \Wcal^{(y)}_{s, s+t}( \cdot -  V(y), \cdot).
	\end{equation}

	The notation introduced so far allows for a probabilistic representation of the solutions to~\eqref{eq:2:KineticChar}.

	\begin{proof}[Proof of Theorem \ref{mthm1}]
		We intend to prove that   the function
		\begin{equation*}
			\psi_t(\xi) = \E \big[ e^{i\xi W_t} \big], \quad t\geq 0, \quad \xi \in \R,
		\end{equation*}
		is the unique solution to equation \eqref{eq:2:KineticChar}. Recall also that $\phi_0$ is the characteristic function of $R$.
		First we prove that the function $\psi_t$ satisfies equation 	\eqref{eq:2:KineticChar}.	
		The argument follows the proof of~\cite[Proposition 2.5]{Bogus2020} with some slight changes. We first establish an auxiliary continuity estimate for $\psi_t$. Take $t>s$ and recall that under our standing assumptions
		$Y_s$ is a.s. finite. Using~\eqref{eq:3:jumpsY} we write
		$$	|\psi_t(\xi) - \psi_s(\xi)| \leq   2 \P \left[ \mbox{there are splits during $[s,t]$} \right]
								\leq   2 \E \big[1 - e^{-(t-s)Y_s} \big] \to 0, \qquad s\to t,
		$$
		by an appeal to the dominated convergence theorem. Now define the filtration $(\Fcal_s)_{s\geq 0}$ via
		\begin{equation*}
			\Fcal_s = \sigma \left( (E(x), C(x), (Z_k(x))_{k \geq 0}\: : \: x \in \Tcal_s \right)
		\end{equation*}
		and use~\eqref{eq:3:decW} to write for $s<t$,
		$$
			\psi_t(\xi) =   \E \left[ \E \left[ \left.  \exp \left\{  i \xi  W_t   \right\} \right| \Fcal_s \right] \right]=
				    \E \bigg[ \E \bigg[ \bigg.  \exp \bigg\{  i \xi C_s + \sum_{y \in \partial \Tcal_s} i \xi e^{V(y)}\int e^{v}z \Wcal^{(y)}_{s, t}(\ud v, \ud z) \bigg\} \Big| \Fcal_s \bigg] \bigg].
		$$
		The  measurability of ${C}_s$ with respect to $\mathcal{F}_s$ yields
		\begin{equation}\label{eq:3:conv}	
				\psi_t(\xi)=   \E \bigg[ e^{ i \xi  C_s} \prod_{y \in \partial \Tcal_s} \psi_{t-s} \big(\xi e^{V(y)}\big) \bigg] .
		\end{equation}
		We shall use this formula to study left and right derivatives of $\psi_t$ with respect to $t$.
		For $h\in(0,t)$, by considering the number of splits that occur within the time interval $[0,h]$, we arrive at
		\begin{align*}
			\psi_t(\xi)  
				= &  \E \bigg[ e^{ i \xi C_h 
} \prod_{y \in \partial \Tcal_h} \psi_{t-h} \big( \xi e^{V(y)}\big)  \ind_{\left\{ \mbox{no splits during $[0,h]$} \right\}} \bigg] \\
				&+  \E \bigg[ e^{ i \xi C_h
} \prod_{y \in \partial \Tcal_h} \psi_{t-h}\big( \xi e^{V(y)}\big)  \ind_{\left\{ \mbox{one split during $[0,h]$} \right\}} \bigg]  +o(h) \\
				= & \psi_{t-h}(\xi)e^{-h} + \E \left[ e^{ i \xi C} \prod_{k=1}^N \psi_{t-h}(\xi A_k)  \right]h +o(h).
		\end{align*}
		Dividing both sides by $h$, sending $h \to 0^+$ and using the continuity estimate established at the beginning of the proof yields
		\begin{equation*}
			\partial^-_t \psi_t ( \xi) + \psi_t(\xi) = \E \bigg[ e^{ i \xi C} \prod_{k=1}^N \psi_t ( \xi A_k)  \bigg] = \widehat{\Qcal}\psi_t (\xi).
		\end{equation*}	
		If we now replace $t$ with $t+h$, put $s=h$ in~\eqref{eq:3:conv}, and then send $h\to 0^+$ we conclude that $\partial^+_t\psi_t = \partial^-_t \psi_t$. Thus, $\psi$ is indeed a solution of~\eqref{eq:2:KineticChar} as claimed.

		To prove uniqueness of the solution we adopt to our settings the arguments provided in the proof of~\cite[Proposition 1.4]{Buraczewski2019}.
		Let $\phi_t$ be any solution to \eqref{eq:2:KineticChar}. Going back to~\eqref{eq:1:prob} write
		\begin{align*}
			\phi_t(\xi)  &= \P[E(\varnothing) >t ]\phi_0(\xi) + \E \left[ \ind_{\{  E(\varnothing) \leq t\}} \widehat{\mathcal{Q}}\phi_{t-E(\varnothing)}(\xi) \right] \\
			 &= \E \left[ \prod_{x \in \partial \mathcal{T}_t, |x|\leq 1 }\phi_0\left(e^{V(x)}\xi \right)  \prod_{ x \in \mathcal{T}_t^o, \:|x| \leq 0} e^{i\xi e^{V(x)}C(x)}\prod_{ x \in \mathcal{T}_t^o, \:|x| = 1}
				\phi_{t-D(x)}\left(e^{V(x)}\xi \right)   \right].
		\end{align*}
		Iterating the above equation $n$ times gives
		\begin{equation*}
			\phi_t(\xi) =\E \left[ \prod_{x \in \partial \mathcal{T}_t, |x|\leq n }\phi_0\left(e^{V(x)}\xi \right)  \prod_{ x \in \mathcal{T}_t^o, \:|x| \leq n-1} e^{i\xi e^{V(x)}C(x)}
			\prod_{ x \in \mathcal{T}_t^o, \:|x| = n} \phi_{t-D(x)}\left(e^{V(x)}\xi \right)   \right].
		\end{equation*}
		Under assumption~\eqref{eq:2:NoEx}, $\partial \mathcal{T}_t$ and $\mathcal{T}_t^o$ are both finite a.s. and therefore all three products under the expectation have a.s.~finite limits as $n \to \infty$.
		Furthermore, the limit of the third product is one.
		An appeal to the dominated convergence theorem gives
		\begin{equation*}
			\phi_t(\xi) =  \E \left[ \prod_{x \in \mathcal{T}_t^o}e^{i\xi e^{V(x)}C(x)} \prod_{ x \in \partial \mathcal{T}_t} \phi_{0}\left(e^{V(x)}\xi \right) \right],\quad t\geq 0,\quad \xi\in\R.
		\end{equation*}
		Recalling that $\phi_0$ is the characteristic function of the random variable $R$ we obtain
		$$
			\phi_t(\xi)  = \E \bigg[ \prod_{x \in \mathcal{T}_t^o}e^{i\xi e^{V(x)}C(x)} \prod_{ x \in \partial \mathcal{T}_t} e^{i\xi e^{V(x)}R(x)} \bigg] =
			\E\big[ e^{i\xi C_t }e^{i\xi X_{R,t}}\big] = \E\big[ e^{i\xi W_t } \big]= \psi_t(\xi)
		$$
		for $ t\geq 0, \xi\in\R$. The proof is complete.
	\end{proof}

\section{Many-to-one lemmas}\label{sec:many2one}

	One of the most useful tools in the analysis of branching random walks are many-to-one lemmas. It turns out that one can express the expectation of a statistic, which depends on the entire collection of particles, in terms of a trajectory of one,
	tagged, particle after an appropriate change of measure. The claims of these lemmas follow from general theory (see~\cite[Theorem 8.2]{hardy2009spine}).
	We shall however present the details for the readers convenience. For $\alpha\in \mathcal{D}$ denote
	\begin{equation*}
		\lambda(\alpha) = \log \EE \bigg[ \sum_{|x|=1} e^{\alpha V(x)} \bigg] = \log\big( \Phi(\alpha) + 1 \big)
	\end{equation*}
	and consider a standard zero-delayed random walk $S^{(\alpha)} = \big(S_n^{(\alpha)}\big)_{n \geq 0}$ with increments having distribution defined by the formula
	\begin{equation}\label{eq:change_of_measure}
		\EE\big[h\big(S^{(\alpha)}_1\big)\big] = \EE \bigg[ \sum_{|x|=1} e^{\alpha V(x)-\lambda(\alpha)}  h(V(x)) \bigg]
	\end{equation}
	for any bounded measurable function $h:\R\to\R$.
	Furthermore, let $\eta=(\eta_t)_{t\geq0}$ be a~homogeneous Poisson process on $[0,\infty)$ given via
	\begin{equation*}
		\eta_t = \# \left\{ k \geq 1 \: : \: T_k \leq t \right\},\quad t\geq 0,
	\end{equation*}
	where $T_k = \sum_{j=1}^k E_j$, $k\in\N$, and $(E_j)_{j \geq 1}$ is a sequence of i.i.d. random variables with the unit mean exponential distribution.  Next lemma is the aforementioned many-to-one formula for Markov
	branching process specialized to the continuous time random walk
	$\big( S_{\eta_t}^{(\alpha)} \big)_{t \geq 0}$, see~\cite[Theorem 8.2]{hardy2009spine}. In what follows $S^{(\alpha)}$, $\eta$ and $U$ are assumed to be independent.

	\begin{lem}[Many-to-one formula I]\label{lem:4:Many2one1}
		Assume~\eqref{eq:2:NoEx},~\eqref{eq:Ai_positive},~\eqref{eq:2.1'} and $\mathcal{D}\neq \varnothing$.
		Under the introduced notation, for any $\alpha \in \mathcal{D}$, $t\geq 0$ and any measurable function $g \colon \R^2 \to [0, \infty)$,
		\begin{equation}\label{eq:4:Many2one1}
			\EE \left[ \int g(v,x) \: \Xcal_{U,t}(\ud v,\ud x)\right] = \EE \Big[ e^{-\alpha S_{\eta_t}^{(\alpha)} + \lambda(\alpha) \eta_t} g\big(S_{\eta_t}^{(\alpha)}, U\big) \Big].
		\end{equation}
	\end{lem}

\begin{proof}
 	Define for $n \in \NN\cup\{0,\infty\}$,
 	\begin{equation*}
		F_n(t,z) =
		\begin{cases}
		\EE \big[ \int g(z+v,x) \: \Xcal_{U,t}^n(\ud v,\ud x)\big], & t\geq 0,\\
		\EE\big[g(z,U)\big], & t < 0,
		 \end{cases}
	\end{equation*}
	where
 	\begin{equation*}
		\Xcal_{U,t}^n = \sum_{ x \in \partial\Tcal_t, |x| < n}\delta_{V(x)} \otimes \delta_{U(x)},\quad  t\geq0.
	\end{equation*}
	Then $\Xcal_{U,t}^\infty  =\Xcal_{U,t}$ and thus $F_\infty(t,0)$ is equal to the left-hand side of~\eqref{eq:4:Many2one1}.
	Note also that $F_0(t,z)=0$ for $t\ge 0$.
	Utilizing the branching property of $\Xcal_{U,t}^n$	
	and recalling \eqref{eq:change_of_measure} we obtain, for finite $n\ge 1$,
	\begin{align*}
		F_n(t,z) & =  \EE \bigg[ \ind_{\{ E(\varnothing) \leq t \}} \sum_{|y|=1}  F_{n-1}(t-E(\varnothing), z+V(y))  + \ind_{\{E(\varnothing) >t \}} g(z,U)\bigg]   \\
			& =  \EE \Big[  e^{-\alpha S_1^{(\alpha)} +\lambda(\alpha)} \ind_{\{ E_1\leq  t \}}    F_{n-1}\big(t-E_1, z+S_1^{(\alpha)}\big) +\ind_{\{ E_1>  t \}}g(z,U   ) \Big] \\
			& = \EE \Big[ \exp \big\{ {-\alpha S_{\eta_t^{(1)}}^{(\alpha)} + \lambda(\alpha) \eta_t^{(1)}} \big\} F_{n-1}\big(t-T_1,z+ S_{\eta_t^{(1)}}^{(\alpha)}\big) \Big],
	\end{align*}	
 	where
 	\begin{equation*}
		\eta_t^{(k)} = \# \left\{ 1\leq j\leq k \: : \: T_j \leq t \right\},\quad t\geq 0,\quad k\in\N.
	\end{equation*}	
	Upon iterating $n$ times the above formula for $F_n(t,z)$ we deduce
	\begin{align*}
		F_n(t,z) &=  \EE \Big[ \exp \big\{ {-\alpha S_{\eta_t^{(n)}}^{(\alpha)} + \lambda(\alpha) \eta_t^{(n)}} \big\} F_0\big(t-T_n,z+ S_{\eta_t^{(n)}}^{(\alpha)}\big) \Big]\\ &=
		\EE \Big[ \exp \big\{ {-\alpha S_{\eta_t}^{(\alpha)} + \lambda(\alpha) \eta_t} \big\} g\big(z+ S_{\eta_t}^{(\alpha)},U\big)\ind_{\{T_n>t\}} \Big].
	\end{align*}
	Putting $z=0$ in the above equality and sending $n\to\infty$ yields~\eqref{eq:4:Many2one1} in view of the monotone convergence theorem.
\end{proof}

	For $t\geq 0$ and $\gamma\in\R$ define
	\begin{equation}\label{eq:cc4}
	\begin{split}
		Z_t(\gamma)&=\int e^{\gamma v}\Zcal_t(\ud v) 	= \sum_{x\in \partial \Tcal_t} e^{\gamma V(x)} .
	\end{split}
	\end{equation}
	A direct consequence of Lemma~\ref{lem:4:Many2one1} is the following formula, valid for $\gamma\in \mathcal{D}$,
	\begin{equation}\label{eq:p5}
		\E \big[Z_t(\gamma)\big] =		\EE \left[ \int e^{\gamma v} \: \Zcal_t(\ud v)\right] = e^{t\Phi(\gamma)},\quad t\geq 0.
	\end{equation}
	We mention in passing that the above formula combined with the branching property of $\mathcal{X}$ stated in~\eqref{eq:3:branchingX} can be used to check that $M_t(\gamma)$ defined in~\eqref{eq:star} is indeed a martingale.
	Differentiating both sides of the above equation with respect to $\gamma$ yields that for $\gamma \in {\rm int} (\mathcal{D})$,
	\begin{equation*}
		\EE \left[ \int ve^{\gamma v} \: \Zcal_t(\ud v)\right] =t \Phi'(\gamma ) e^{t\Phi(\gamma)}
		\quad \mbox{and} \quad 		\EE \left[ \int v^2e^{\gamma v} \: \Zcal_t(\ud v)\right] = \left( t\Phi''(\gamma) + t^2(\Phi'(\gamma))^2\right) e^{t\Phi(\gamma)}.
	\end{equation*}	
	We shall also use a simple consequence of the Marcinkiewicz--Zygmund inequality, see Theorem 10.3.2 in \cite{chow:teicher:88}). If $\EE[U]=0$, then, for any $p\in[1,2]$ and $\gamma>0$ such that $p\gamma \in\mathcal{D}$,
	there exists a universal constant $c_p>0$ that depends only on $p$ such that
	\begin{equation}\label{eq:4:MZ}
 	\begin{split}
 		\E[|X_{U,t}|^p]  &= \EE \left[ \left| \int e^{\gamma v} u \Xcal_{U,t}(\ud v, \ud u) \right|^p \right]   \leq c_p\EE \left[ \left( \int e^{2\gamma v} u^2 \Xcal_{U,t}(\ud v, \ud u) \right)^{\frac p2} \right]\\
			 &\leq c_p\EE \bigg[ \int e^{p\gamma v} |u|^p \Xcal_{U,t}(\ud v, \ud u)  \bigg]
 			= c_p\EE \left[ \int e^{p\gamma v}  \Zcal_{t}(\ud v)  \right] \EE \left[ |U|^p \right]
			= c_p  e^{t\Phi(p\gamma)} \EE \left[ |U|^p \right].
	\end{split}
	\end{equation}
	In the above chain of estimates the first one is the aforementioned Marcinkiewicz--Zygmund inequality, the second one follows from subadditivity of $x\mapsto |x|^{p/2}$, for $p\in [1,2]$, and the final equality follows from~\eqref{eq:p5}.

	Our next lemma is a less standard one. In view of similarities with the previous result we shall also refer to it as a many-to-one formula.
	\begin{lem}[Many-to-one formula II]\label{lem:4:Mant2one2}
		Assume~\eqref{eq:2:NoEx},~\eqref{eq:Ai_positive},~\eqref{eq:2.1'} and $\mathcal{D}\neq \varnothing$.  For every $\alpha \in \mathcal{D}$, every measurable function $g \colon \R^2 \to [0, \infty)$ and every $t \geq 0$,
		\begin{equation}\label{eq:4:Many2one2}
			\EE \left[ \int g(v,c) \: \Ccal_t(\ud v,\ud c)\right] =  \EE \bigg[  \sum_{k=0}^{\eta_t-1} e^{-\alpha S_{k}^{(\alpha)} + \lambda(\alpha) k} g\big(S_{k}^{(\alpha)}, C\big) \bigg],
		\end{equation}
		where $C$ is independent of $S^{(\alpha)}$ and $\eta$ on the right-hand side.
	\end{lem}

\begin{proof}
	We shall use the same procedure as in the proof of Lemma~\ref{lem:4:Many2one1}.
	Define, for $n \in \NN\cup \{0, \infty\}$,
	\begin{equation*}
		H_n(t,z) = \begin{cases}
		 \EE \left[ \int g(z+v,c) \: \Ccal_t^n(\ud v,\ud c)\right], & t\geq 0, \\
		 0, & t < 0,
		 \end{cases}
	\end{equation*}
	where
	\begin{equation*}
		\Ccal_t^n := \sum_{ x \in \Tcal_t^o, |x|<n}\delta_{ V(x) } \otimes \delta_{C(x)},\quad t\geq 0.
	\end{equation*} In particular, $H_0(t,z)=0$. As in the previous lemma we shall use the branching property of $\Ccal_t^n$. For a finite $n\in\NN$, it holds
	\begin{align*}
		H_n(t,z) &=  \EE \bigg[ \sum_{|y|=1}  H_{n-1}(t-E(\varnothing), z+V(y)) + g(z, C) \ind_{\{ E(\varnothing) \leq t \}}\bigg]   \\
			 &=  \EE \Big[  e^{-\alpha S_1^{(\alpha)} +\lambda(\alpha)}    H_{n-1}\big(t-E_1, z+S_1^{(\alpha)}\big) +\ind_{\{ E_1\leq t \}}g(z,C) \Big].
	\end{align*}
	Iterating the above equality yields
	\begin{align*}
		H_n(t,z) &= \EE \bigg[  \exp \big\{-\alpha S_{\eta_t^{(n)}}^{(\alpha)} +\lambda(\alpha)\eta_t^{(n)} \big\}    H_0\big(t-T_n, z+S_{n}^{(\alpha)}\big) + \sum_{k=0}^{\eta_t^{(n)}-1} e^{-\alpha S_{k}^{(\alpha)}
			+ 	\lambda(\alpha) k} g\big(z+S_{k}^{(\alpha)}, C\big)   \bigg]\\
 			&= \EE \bigg[ \sum_{k=0}^{\eta^{(n)}_t-1} e^{-\alpha S_{k}^{(\alpha)} + \lambda(\alpha) k} g\big(z+S_{k}^{(\alpha)}, C\big)   \bigg].
	\end{align*}
	Taking $z=0$ and sending $n\to\infty$ yields the desired formula by an appeal to the monotone convergence theorem.
\end{proof}

	Assume that $\E [|C|^p]<\infty$ for some $p>0$. Then formula~\eqref{eq:4:Many2one2} entails
	\begin{equation}\label{eq:c_mixed_moment}
		\EE \left[ \int e^{\gamma v} |c|^p \: \Ccal_t(\ud v,\ud c)\right] =
		\EE\left[|C|^p\right]  \cdot
		\begin{cases}
		\frac{e^{t\Phi(\gamma)} -1}{\Phi(\gamma)},  & \Phi(\gamma)\neq 0, \\
		t, & \Phi(\gamma)=0,
		\end{cases}
	\end{equation}
for every $\gamma\in\mathcal{D}$. Recall that $C_t$ is defined via~\eqref{eq:3:Ct} and note that, given $\mathcal{T}_t^{o}$ and $(V(x))_{x\in \mathcal{T}_t^{o}}$, the variables $(C(x))_{x\in\mathcal{T}_t^{o}}$ are independent copies of $C$. Thus, if $\E [C]=0$, we can apply the Marcinkiewicz--Zygmund inequality as we did in~\eqref{eq:4:MZ}. This entails,
for $q\in [1,2]$,
\begin{equation}\label{eq:c3}
  \E\big[ |C_t|^q \big] \le c_q 	\EE \left[ \int e^{q v} |c|^q \: \Ccal_t(\ud v,\ud c)\right].
\end{equation}

\section{Limits of the associated stochastic processes: proofs of Theorems \ref{mthm2} and \ref{mthm3} }\label{sec:Conv}

	Having represented the solution to~\eqref{eq:2:KineticChar} as the characteristic function of the process $(W_t)_{t\ge 0}$
we shall analyse the asymptotic behaviour of the latter.
Recall the decomposition $W_t = C_t + X_{R,t}$ with $X_{R,t}$ and $C_t$ given via \eqref{eq:3:Xt} and ~\eqref{eq:3:Ct}, respectively.
The asymptotics of $(X_{R,t})_{t\ge 0}$ is essentially determined by the condition $(H_{\gamma})$ and the behaviour of the process $(M_t(\gamma))_{t\ge 0}$, defined in \eqref{eq:star}, which is well-known from the literature, since
\begin{equation*}
	M_t(\gamma) = e^{-\Phi(\gamma) t}\int e^{\gamma v}\Zcal_t(\ud v) =  e^{-\Phi(\gamma) t}Z_t(\gamma)
\end{equation*}
forms a so-called continuous-time additive martingale. Since it is positive, it converges a.s.~to a limit as $t\to \infty$. It turns out that the limit is nondegenerate if and only if $\mu'(\gamma) < 0$. Whenever $\mu'(\gamma) \ge 0$, the martingale $M_t(\gamma)$ converges to $0$ a.s.~and a polynomial correction term is necessary to obtain a nondegenerate limit for a rescaled $M_t(\gamma)$. We sum up this discussion in the following lemma, which follows from \cite[Theorem 1.1]{Bertoin:2018:Biggins}.
\begin{lem}\label{lem:5:MartConvSub}
		Assume~\eqref{eq:2:NoEx},~\eqref{eq:Ai_positive},~\eqref{eq:2.1'} and~\eqref{eq:2:super_crit}. Suppose further that $\mathcal{D}\neq \varnothing$ and
$$
\EE \left[ \left(\sum_{k=1}^N A_k^\gamma\right) \log^+\left(\sum_{k=1}^N A_k^\gamma\right)  \right] <\infty,
$$
for some $\gamma \in \mathcal{D}$. If $\mu'(\gamma)<0$, then
	\begin{equation*}
		\lim_{t\to+\infty}M_t(\gamma) = M_\infty(\gamma) \quad \mbox{a.s. and in }L_1.
	\end{equation*}
	Furthermore, $\PP[M_\infty (\gamma) >0\; |\; |\mathcal{T}_\infty| = \infty ] =1$ and also
	\begin{equation}\label{eq:MartConvSub_max}
		\lim_{t\to+\infty}\left(\max_{x \in \partial \mathcal{T}_t}\gamma V(x) -\Phi(\gamma)t\right)= -\infty\quad\text{a.s.}
	\end{equation}
\end{lem}

Note that~\eqref{lem:5:MartConvSub} holds a.s.~conditionally on the event $\{|\mathcal{T}_\infty| = \infty\}$. However, on the complementary event $\{|\mathcal{T}_\infty| < \infty\}$, the set $\partial \mathcal{T}_t$ is empty for all large enough $t>0$, and~\eqref{eq:MartConvSub_max} is still true if we interpret the maximum over the empty set as $-\infty$.

We aim now to prove Theorem~\ref{mthm2}. To this end we apply a similar procedure as in the proofs of Proposition~3.2 in~\cite{Bogus2020} and Theorem~2.3 in~\cite{dadoun:2017}. First, we fix some $\theta>0$
and analyse the asymptotic behavior of the discrete skeleton $(C_{\theta n})_{n\ge 0}$. At the second step we pass to a continuous parameter by using an appropriate approximation argument. Note that for $\theta>0$ and $n\in\N$, branching property~\eqref{eq:3:branchingC} yields
\begin{equation*}
	 C_{\theta (n+1)} =  C_{\theta n} + \sum_{x \in \partial \mathcal{T}_{\theta n}} e^{ V(x)} \int e^{ v} c \:\Ccal^{(x)}_{\theta n,\theta (n+1)}(\ud v,\ud c),
\end{equation*}
and $C_{\theta}^{(n)}(x):=\int e^{ v} c \:\Ccal^{(x)}_{\theta n,\theta (n+1)}(\ud v,\ud c)$, $x \in \partial \mathcal{T}_{\theta n}$, are conditionally independent copies of $C_{\theta}$. Thus, recalling \eqref{eq:cc4},
\begin{equation}\label{eq:c4}
\begin{split}
C_{\theta (n+1)} &= C_{\theta n} + \sum_{x \in \partial \mathcal{T}_{\theta n}} e^{V(x)} C_{\theta}^{(n)}(x)\\
&= C_{\theta n} + (\E [C_{\theta}]) \cdot Z_{\theta n}(1) + \sum_{x \in \partial \mathcal{T}_{\theta n}} e^{V(x)} \big( C_{\theta}^{(n)}(x) - \E [C_{\theta}]\big).
\end{split}
\end{equation}
Denote
$$ \XX_k := \sum_{x \in \partial \mathcal{T}_{\theta k}} e^{V(x)} \big( C_{\theta}^{(k)}(x) - \E [C_{\theta}]\big),
$$
so that
\begin{equation}\label{eq:c1}
	 C_{\theta (n+1)} =  \E[C_{\theta}] \sum_{k=0}^n Z_{\theta k}(1) + \sum_{k=0}^n \XX_k.
\end{equation}

\begin{proof}[Proof of Theorem \ref{mthm2} in cases (a) and (b)]
	We start with decomposition \eqref{eq:c1}. The first term on the right-hand side of~\eqref{eq:c1} can be treated using Lemma~\ref{lem:5:MartConvSub}. Indeed, since $\mu'(1)<0$,
	\begin{equation}\label{eq:5:auxConv}
		\lim_{n\to\infty}e^{-\theta n \Phi(1)} Z_{\theta n}(1) = \lim_{n\to\infty} M_{\theta n}(1) = M_\infty(1)\quad\text{a.s.}
	\end{equation}
	To estimate the second term in~\eqref{eq:c1}, pick a constant $q\in (1,2]\cap {\rm int}(\mathcal{D})$ such that $\E |C|^q <\infty$ and $\mu'(q)<0$, then $\mu(1) > \mu(q)$.
	Notice that $\XX_k$ has the same distribution as $X_{U^{(k)},\theta k}$ with the vertices marks $U^{(k)}(x)=C_{\theta}^{(k)}(x) - \E [C_{\theta}]$, $x\in\partial \mathcal{T}_{\theta k}$.
	Inequalities \eqref{eq:c_mixed_moment}, \eqref{eq:c3} yield
	$$
		\E \big[\big| C_{\theta}^{(n)}(x) - \E [C_{\theta}]  \big|^q\big] < c_{q,\theta},
	$$
	for some constant $c_{q,\theta}$. Thus, \eqref{eq:4:MZ} entails
	$$
		\E [| \XX_n |^q] \le c_{q,\theta} e^{\theta n \Phi(q)}.
	$$
	For any fixed $\varepsilon>0$ we can write using Markov's inequality
	$$
		\PP\left[ \big| \XX_n  
		\big|>\varepsilon e^{n\theta\Phi(1)}\right]
		\leq \varepsilon^{-q} e^{-q n\theta \Phi(1)}\EE [| \XX_n  
|^q]
\leq c_{q,\theta}  \varepsilon^{-q}
\cdot e^{q  n\theta (\mu(q)-\mu(1) ) }.
	$$
	A standard application of the Borel-Cantelli lemma entails
	\begin{equation}\label{eq:bc}
		e^{- n\theta \Phi(1)}  \XX_n  
		\to 0\quad\text{a.s.}
	\end{equation}
	as $n \to \infty$.
	We now argue case by case.

\textsc{Case $(a)$}: If $\Phi(1) >0$, then $\EE[C_{\theta}] = \EE [C] (e^{\theta\Phi(1)}-1)/\Phi(1)$ by \eqref{eq:c_mixed_moment}.
Using the Stolz--Ces\'{a}ro lemma it is easy to check that the limit relation $e^{-\beta k} a_k \to a_\infty$, $k\to\infty$, for some sequence of real numbers $(a_n)_{n\in \N}$, $\beta>0$ and $a_{\infty}\in\R$, implies $$e^{-\beta n} \sum_{k=1}^n a_k \to\frac{e^{\beta}}{e^{\beta}-1}a_{\infty},\qquad n\to\infty.$$ Therefore, invoking \eqref{eq:c1} together with  \eqref{eq:5:auxConv} and~\eqref{eq:bc},
\begin{equation*}
	e^{-\theta n \Phi(1)}C_{\theta (n+1)} = e^{-\theta n \Phi(1)}\sum_{k=0}^n  \bigg(\EE[C] \frac{e^{\theta\Phi(1)} -1}{\Phi(1)}  Z_{\theta k}(1) +
 	\XX_k\bigg) \to\frac{e^{\theta\Phi(1)} \EE[C]}{\Phi(1)} M_\infty(1)\quad\text{a.s.}
\end{equation*}

\textsc{Case $(b)$}: If $\Phi(1)=0$, then $\EE[C_{\theta}] = \theta\EE[C]$ by \eqref{eq:c_mixed_moment}. Moreover, since $Z_{\theta n}(1)\to M_{\infty}(1)$ a.s. by \eqref{eq:5:auxConv} and $\XX_n\to 0$ a.s. by~\eqref{eq:bc}, another appeal to  the Stolz--Ces\'{a}ro lemma yields
	\begin{equation*}
		\frac 1{n\theta}C _{(n+1)\theta} \to \EE[C] M_\infty(1)\quad\text{a.s.}
	\end{equation*}

To conclude the proof of Theorem~\ref{mthm2} in cases (a) and (b) it remains to pass from the discrete parameter $n\theta$ to the continuous parameter $t$ in both cases. Take $t>0$ and for $\theta>0$ let $n = \lfloor {t}/{\theta} \rfloor$. Utilizing~\eqref{eq:3:branchingC}, we conclude
$$
\E [|	C_{t} - C_{\theta n}|] = \E [|X_{ C_{t - \theta n},\theta n}|].$$
Since, invoking \eqref{eq:c_mixed_moment},
$$
\E [| C_{t - \theta n}|]  \le   \E \left[\int e^{ v}|c| \:\Ccal_{ \theta }(\ud v, \ud c)\right]  \leq {\rm const}\cdot \theta \E [|C|] <\infty,$$
we can bound the $L_1$ norm of $C_{t} - C_{\theta n}$ as follows
\begin{equation*}
	\EE \left[ \left| C_{t} - C_{\theta n}  \right| \right] \leq
\E[Z_{\theta n} (1)] \E | C_{t - \theta n}| \le
{\rm const}\cdot  \theta  e^{\theta n \Phi(1)} \EE[|C|].
\end{equation*}
This implies $$\lim_{\theta\to 0^+} \limsup_{t \to \infty}e^{-\theta n \Phi(1)}\EE \left[ \left| C_{t} - C_{\theta n} \right| \right]  =0,$$ which  secures our claim.
\end{proof}

\begin{proof}[Proofs of Theorem \ref{mthm2} in case (c) and of Theorem \ref{mthm3} in case (d)]
Fix $\theta>0$. As in the previous part, we shall first show the convergence along the sequence $(\theta n)_{n \geq 0}$. We shall argue case by case.

 	\textsc{Case} (d): If $\E [C] = 0$, then \eqref{eq:c4} entails that $(C_{\theta n})_{n \geq 0}$ is a martingale.
	It is uniformly integrable since for $p \in (1,2]\cap \mathcal{D}$ and $\Phi(p)<0$, invoking consecutively
\eqref{eq:c4} and \eqref{eq:4:MZ}, we obtain

$$		
\EE  \left[ |C_{\theta(n+1)} - C_{\theta n}  |^p \right]  =
\E \big[ |\XX_n 
|^p \big]
\le	c_p \EE[|C_\theta|^p] e^{ n \theta \Phi(p)}.
$$
As a consequence, the martingale $(C_{\theta n})_{n \geq 0}$ is bounded in $L_p$ and whence uniformly integrable and convergent.

	\textsc{Case} (c):
	Recall decomposition \eqref{eq:c1}.
	By an appeal to~\eqref{eq:5:auxConv}  the series $\sum_{k=0}^{\infty} Z_{k\theta}(1)$ is absolutely convergent a.s.
	Moreover, the process $\sum_{k=0}^n \XX_k$
forms a uniformly integrable martingale, which can be
	verified in the same way as in case (d).

	In both cases we have shown
	\begin{equation*}
		C_{n\theta}\to C_\infty =  \int e^{v}c\: \Ccal_{\infty}(\ud v, \ud c)  =  \sum_{x \in \mathcal{T}_\infty} e^{V(x)}C(x).
	\end{equation*}
	The convergence $C_t \to C_{\infty}$, $t\to\infty$, can be checked in a similar fashion as in the proof of  Theorem \ref{mthm2}.
\end{proof}

Recall that case (e) of Theorem~\ref{mthm3} follows from the proof of part (e) in Theorem~\ref{mthm4} given in the next section.

\section{Proof of Theorem~\ref{mthm4} and Theorem~\ref{mthm6}}\label{sec:proof}

The convergence results obtained in the previous section allow us to prove Theorem~\ref{mthm4}. We shall use an inhomogeneous version of Lemma 4.1 in~\cite{Bogus2020}.

\begin{lem}\label{lem:6:char}
	Assume that $(r_t)_{t \geq 0}$ is an integer-valued stochastic process such that $r_t \overset{\P}{\to} \infty$, as $t\to\infty$. Suppose additionally that for every $t\geq 0$ there is an array $(a_{k, t})_{k=1}^{r_t}$ of a.s. positive random weights and a stochastic process $(b_t)_{t \geq 0}$ such that
	\begin{equation*}
		\sum_{k=1}^{r_t} a_{k,t}^{\gamma} \overset{\P}{\to} a_{\infty}, \quad \max_{k=1, \ldots , r_t}a_{k, t} \overset{\P}{\to} 0 \quad \mbox{and}\quad  b_t \overset{\P}{\to} b_\infty,\quad t\to\infty,
	\end{equation*}
	for some $\gamma\in (0,2]$, a nonnegative random variable $a_{\infty}$ and a random variable $b_{\infty}$. Assume that $\nu_0$ satisfies the condition $(H_\gamma)$ with the same $\gamma$ as above. For a sequence $(R_k)_{k \in \NN}$ of i.i.d. random variables distributed according to $\nu_0$ and which is independent of $(r_t)_{t \geq 0}$, $\left( (a_{k, t})_{k=1}^{r_t} \: : t\geq 0 \right)$ and $(b_t)_{t \geq 0}$, put
	\begin{equation*}
		S_t = \sum_{k=1}^{r_t} a_{k, t}R_k +b_t.
	\end{equation*}
	Then, as $t \to \infty$,
	\begin{equation}\label{lem:convergence_inhom}
		\EE \left[ e^{i \xi S_t} \right] \to \EE \left[  \hat{g}_\gamma(\xi a_\infty^{1/\gamma}) e^{i\xi b_\infty} \right],\quad \xi\in\R.
	\end{equation}
\end{lem}
\begin{proof}
Note that, for every fixed $\xi\in\R$,
$$
\E \exp (i\xi S_t)=\E \exp\left(i\xi b_t+\sum_{k=1}^{r_t}\log \phi_0(a_{k,t}\xi)\right),
$$
where $\log \phi_0$ is the principal branch of the logarithm of the characteristic function of $R_1$. In Lemma 4.1 in~\cite{Bogus2020} it was proved that
$$
\exp\left(\sum_{k=1}^{r_t}\log \phi_0(a_{k,t}\xi)\right) \overset{\P}{\to} \hat{g}_{\gamma}(\xi a_{\infty}^{1/\gamma}),\quad t\to\infty.
$$
Since $b_t\overset{\P}{\to} b_\infty$, $t\to\infty$, we also have
$$
\exp\left(\sum_{k=1}^{r_t}\log \phi_0(a_{k,t}\xi)+i\xi b_{t}\right) \overset{\P}{\to} \hat{g}_{\gamma}(\xi a_{\infty}^{1/\gamma})e^{i\xi b_{\infty}},\quad t\to\infty.
$$
By the Lebesgue dominated convergence theorem, which is applicable because
$$
\left|\exp\left(\sum_{k=1}^{r_t}\log \phi_0(a_{k,t}\xi)+i\xi b_{t}\right)\right|\leq 1,
$$
we obtain~\eqref{lem:convergence_inhom}.
\end{proof}

\begin{proof}[Proof of Theorem \ref{mthm4}] The first three cases (A), (C) and (D) follow immediately from the above lemma. For the reader convenience we present here the proof of (C). By Lemma \ref{lem:5:MartConvSub} and using that $\Phi(\gamma)=0$, we obtain
$$
\sum_{y\in \partial \Tcal_t} e^{\gamma V(y)} \to M_{\infty}(\gamma)\quad \mbox{a.s.}\quad \mbox{ and } \quad
\max_{y\in \partial \Tcal_t} e^{V(y)} \to 0\quad \mbox{a.s.,} \qquad t\to\infty.
$$
Further, by Theorem~\ref{mthm2}(c), $C_t\overset{\P}{\to} C_{\infty}$, $t\to\infty$. Therefore, by  Lemma \ref{lem:6:char} applied with $r_t = Y_t = |\partial \Tcal_t|$, $a_{y,t} = e^{\gamma V(y)}$, $b_t = C_t$, it holds
$$
\E \big[ e^{i\xi W_t} \big] \to \E\big[\hat{g}_\gamma   (\xi M_{\infty}^{1/\gamma}(\gamma)) e^{i\xi C_\infty} \big],\quad t\to\infty.
$$

Let us now consider case (E). Fix $\theta>0$ and let us show that
\begin{equation}\label{eq:mthm4:proof1}
e^{-\mu(2)\theta n}W_{\theta n}\overset{d}{\to} \sqrt{M_{\infty}(2)(\sigma_0^2+\E [C]^2/\Phi(2))}{\rm N},\quad n\to\infty,
\end{equation}
where ${\rm N}$ is a standard normal random variable which is independent of $M_{\infty}(2)$.

Fix $l\in\N$. Using the decomposition $W_{\theta n}=X_{R,\theta n}+C_{\theta n}$ and the branching properties~\eqref{eq:3:branchingX} and~\eqref{eq:3:branchingC}, we can write
\begin{align*}
e^{-\mu(2)\theta n}W_{\theta n}&=e^{-\mu(2)\theta n}(X_{R,\theta n}+C_{\theta n})\\
&=\left(e^{-\mu(2)\theta l}\sum_{y\in\partial \mathcal{T}_{\theta(n-l)}}e^{V(y)-\mu(2)\theta (n-l)}\left(X_{R,\theta l}^{(y)}+C_{\theta(n-l),\theta n}^{(y)}\right)\right)+e^{-\mu(2)\theta n}C_{\theta(n-l)},
\end{align*}
where
$$
C_{\theta(n-l),\theta n}^{(y)}:=\int e^v c\mathcal{C}^{(y)}_{\theta (n-l),\theta n}({\rm d}v,{\rm d}c)\quad\text{and}\quad X_{R,\theta l}^{(y)}:=\int e^v r \mathcal{X}_{R,\theta l}^{(y)}({\rm d}v,{\rm d}r),\quad y\in\partial \mathcal{T}_{\theta(n-l)}.
$$
Note that, for every $y\in\partial \mathcal{T}_{\theta(n-l)}$, $X_{R,\theta l}^{(y)}+C_{\theta(n-l),\theta n}^{(y)}$ are independent copies of a random variable
$$
W_{\theta,l}:=X_{R,\theta l}+C_{\theta l},
$$
and are also independent of $\partial \mathcal{T}_{\theta(n-l)}$ and $(e^{V(y)})_{y\in \partial\mathcal{T}_{\theta(n-l)}}$. Since
$$
\E [W_{\theta,l}] = 0\quad\text{and}\quad \E[W_{\theta,l}^2]<\infty,
$$
we can apply Lemma~\ref{lem:convergence_inhom} with $\gamma=2$, $t=\theta(n-l)$, $a_{y,t}:=e^{V(y)-\mu(2)t}$, $r_{t}=|\partial \mathcal{T}_t|$, $a_\infty=M_{\infty}(2)$ and $b_t=0$, to conclude that
$$
\sum_{y\in\partial \mathcal{T}_{\theta(n-l)}}e^{V(y)-\mu(2)\theta (n-l)}\left(X_{R,\theta l}^{(y)}+C_{\theta(n-l),\theta n}^{(y)}\right)\overset{d}{\to} \sqrt{M_{\infty}(2)\E[W_{\theta,l}]^2}{\rm N},\quad n\to\infty,
$$
for every fixed $l\in\N$. Thus, according to Theorem 3.2 in~\cite{billingsley}, in order to deduce~\eqref{eq:mthm4:proof1} it remains to prove that
\begin{equation}\label{eq:mthm4:proof2}
\lim_{l\to\infty}e^{-\Phi(2)\theta l}\E [W_{\theta,l}^2]=\sigma_0^2+\E [C^2]/\Phi(2),
\end{equation}
and that, for every $\varepsilon>0$,
\begin{equation}\label{eq:mthm4:proof3}
\lim_{l\to\infty} \limsup_{n\to\infty} \P\big\{ \big| e^{-\mu(2)\theta n} C_{\theta(n-l)}    \big| > \varepsilon   \big\} = 0.
\end{equation}
By Chebyshev's inequality in order to prove~\eqref{eq:mthm4:proof3} it suffices to check that
$$
\lim_{l\to\infty} \limsup_{n\to\infty} e^{-\theta n\Phi(2)}\E[ C_{\theta(n-l)}^2] = 0.
$$
But this follows from the chain of estimates
$$
e^{-\theta n\Phi(2)}\E[ C_{\theta(n-l)}^2]\overset{\eqref{eq:c3}}{\leq} c_2e^{-\theta n\Phi(2)} \E\left[\int e^{2v}c^2\mathcal{C}_{\theta(n-l)}({\rm d}v,{\rm d}c)\right]\overset{\eqref{eq:c_mixed_moment}}{=} c_2\E [C^2]e^{-\theta n\Phi(2)}\frac{e^{\theta(n-l)\Phi(2)}-1}{\Phi(2)},
$$
since $\Phi(2)>0$. Limit relation~\eqref{eq:mthm4:proof2} can be checked by direct calculations:
\begin{align*}
\E  W_{\theta,l}^2=\E (X_{R,\theta l}+C_{\theta l})^2&=\E\left(\sum_{x\in\partial\mathcal{T}_{\theta l}}e^{V(x)}R(x)+\sum_{y\in\mathcal{T}^{o}_{\theta l}}e^{V(y)}C(y)\right)^2\\
&=\E\left(\sum_{x\in\partial\mathcal{T}_{\theta l}}e^{V(x)}R(x)\right)^2+\E\left(\sum_{y\in\mathcal{T}^{o}_{\theta l}}e^{V(y)}C(y)\right)^2\\
&=\E [R^2]\E\left[\sum_{x\in\partial\mathcal{T}_{\theta l}}e^{2V(x)}\right]+\E[C^2]\E\left[\sum_{y\in\mathcal{T}^{o}_{\theta l}}e^{2V(y)}\right]\\
&\overset{\eqref{eq:c_mixed_moment}}{=}\sigma_0^2 e^{\theta l \Phi(2)}+\E [C^2] \frac{e^{\theta l \Phi(2)}-1}{\Phi(2)}.
\end{align*}
This completes the proof of~\eqref{eq:mthm4:proof2} as well as of~\eqref{eq:mthm4:proof1}. Fix $\xi\in\R$ and put $h_{\xi}(t):=\E \exp (i\xi e^{-\mu(2)t}W_t)$. From~\eqref{eq:mthm4:proof1} we
know that, for every fixed $\theta>0$, it holds
$$
\lim_{n\to\infty}h_{\xi}(n\theta)=\E \exp\left\{-\xi^2\left(M_{\infty}(2)(\sigma_0^2+\E [C^2]/\Phi(2)\right)/2\right\}.
$$
In order to conclude that the limit $\lim_{n\to\infty}h_{\xi}(n\theta)=\lim_{t\to\infty,t\in\R}h_{\xi}(t)$ it remains to apply the Croft-Kingman lemma, see Corollary 2 in \cite{Kingman}. The fact that $t\mapsto h_{\xi}(t)$ is right-continuous follows from the Lebesgue dominated convergence theorem and the fact that $(W_t)_{t\geq 0}$ has a version with a.s. right-continuous paths.
\end{proof}

\begin{proof}[Proof of Theorem~\ref{mthm6}]
The result is a consequence of the Marcinkiewicz-Zygmund inequality~\eqref{eq:4:MZ}. Let us provide the details in the first case when $(A)$ holds, $\E [R] = 0$ and $\E [|R|^{1+\delta}]<\infty$ for some $\delta > 0$. Invoking Theorem \ref{mthm2}, it is sufficient to ensure that $e^{-\mu(1)t}X_{R,t}$ converges
to 0 in probability. The latter follows from \eqref{eq:4:MZ} applied with $\gamma=1$, because for any $\varepsilon>0$, choosing $p\in (1,1+\delta)\cap \mathcal{D}$ such that $\mu(p)<\mu(1)$, we obtain
$$
\P\big[e^{-\mu(1)t}|X_{R,t}| > \varepsilon\big] \le (\varepsilon e^{\mu(1)t})^{-p}\E [|X_{R,t}|^p] \overset{\eqref{eq:4:MZ}}{\le}
c_p \varepsilon^{-p} \E [|R|^p] \cdot e^{pt(\mu(p) - \mu(1))} \to 0, \quad t\to\infty.
$$
All the remaining cases can be treated in exactly the same way.
\end{proof}

\section*{Acknowledgments}

The authors would like to express their sincere gratitude to two anonymous referees for numerous suggestions, inspiring comments and pointing out to blunders of us. In particular, we are very indebted for providing an exhaustive list of references which were missing in the first version of the manuscript.

\bibliographystyle{amsplain}
\bibliography{SSSofKTEwithPC}

\vspace{1cm}

\footnotesize

\textsc{Dariusz Buraczewski and Piotr Dyszewski:} Mathematical Institute, University of Wroclaw, 
50-384 Wroclaw, Poland\\
\textit{E-mail}: \texttt{dariusz.buraczewski@math.uni.wroc.pl;\
piotr.dyszewski@math.uni.wroc.pl}

\bigskip

\textsc{Alexander Marynych:} Faculty of Computer Science and Cybernetics, Taras Shev\-chen\-ko National University of Kyiv, 01601 Kyiv, Ukraine\\
\textit{E-mail}: \texttt{marynych@unicyb.kiev.ua}

\end{document}